\documentclass[pdflatex]{sn-jnl}% Math and Physical Sciences Numbered Reference Style 
\usepackage{graphicx}%
\usepackage{multirow}%
\usepackage{amsmath,amssymb,amsfonts}%
\usepackage{amsthm}%
\usepackage{mathrsfs}%
\usepackage[title]{appendix}%
\usepackage{xcolor}%
\usepackage{textcomp}%
\usepackage{manyfoot}%
\usepackage{booktabs}%
\usepackage{algorithm}%
\usepackage{algorithmicx}%
\usepackage{algpseudocode}%
\usepackage{listings}%
\usepackage[all,cmtip]{xy}
\usepackage{relsize}
\usepackage{enumerate}
\usepackage{fancyhdr}
\usepackage[numbers,sort&compress]{natbib}
\usepackage{graphics}
\usepackage{relsize}
\usepackage{tikz}
\usepackage{tikz-cd}
\usetikzlibrary{arrows}
\usetikzlibrary{snakes}
\usetikzlibrary{shapes}
\usepackage{fdsymbol}
\usepackage{upgreek}
\usepackage{soul}
\usepackage{url}
\newcommand{\agot}{\mathfrak{a}}
\newcommand{\wneg}{\mathbf{W}}

\newtheorem{theorem}{Theorem}[section]% meant for sectionwise numbers
\newtheorem{proposition}[theorem]{Proposition}% 

\newtheorem{example}[theorem]{Example}%
\newtheorem{remark}[theorem]{Remark}%
\newtheorem{definition}[theorem]{Definition}%
\newtheorem{lemma}[theorem]{Lemma}
\newtheorem{corollary}[theorem]{Corollary}

\begin{document}

\title[Galois Connections and Preradicals in Abelian Categories]{Galois Connections and Preradicals in Abelian Categories}


\author*[1]{\fnm{Rogelio} \sur{Fern\'andez-Alonso}}\email{rfg@xanum.uam.mx}
\equalcont{All authors contributed equally to this work.}

\author[2]{\fnm{Janeth} \sur{Maga\~na}}\email{jamz@azc.uam.mx}
\equalcont{All authors contributed equally to this work.}

\author[1]{\fnm{Martha Lizbeth Shaid} \sur{Sandoval-Miranda}}\email{marlisha@xanum.uam.mx}
\equalcont{All authors contributed equally to this work.}

\author[3]{\fnm{Valente} \sur{Santiago-Vargas}}\email{valente.santiago@ciencias.unam.mx}
\equalcont{All authors contributed equally to this work.}

\affil*[1]{\orgdiv{
Departamento de Matem\'aticas}, \orgname{Universidad Aut\'onoma Metropolitana - Iztapalapa}, \orgaddress{\street{Av. San Rafael Atlixco 186, Col. Vicentina Iztapalapa}, \city{Ciudad de M\'exico}, \postcode{09340}, \state{CDMX}, \country{M\'exido}
}}

\affil[2]{\orgdiv{Departamento de Ciencias B\'asicas CBI}, \orgname{Universidad Aut\'onoma Metropolitana - Azcapotzalco}, \orgaddress{\street{Av. San Pablo No. 180, Col. Reynosa Tamaulipas}, \city{Ciudad de M\'exico}, \postcode{02200}, \state{CDMX}, \country{M\'exico}}}

\affil[3]{\orgdiv{Departamento de Matem\'aticas}, \orgname{Facultad de Ciencias, Universidad Nacional Aut\'onoma de M\'exico}, \orgaddress{\street{Circuito Exterior, Ciudad Universitaria}, \city{Ciudad de M\'exico}, \postcode{04510}, \state{CDMX}, \country{M\'exico}}}

\abstract{In this paper we define operations of preradicals of any abelian category. We define idempotent preradicals and radicals. We prove that every adjoint pair between abelian categories induces a Galois connection between the corresponding ordered collections of preradicals. If the abelian categories are bicomplete, we construct alpha and omega preradicals and study their respective preservation under the Galois connection. If a bicomplete abelian category is in addition locally small, then the corresponding collection of preradicals is a complete lattice. The Galois connection induced by an adjoint pair between locally small bicomplete abelian categories preserves, respectively, idempotent preradicals and radicals.}



\maketitle
\section{Introduction}

In \cite{BJKN, BKN}, Bican, Kepka, N\~emec and Jambor established that every equivalence between categories of modules $R\mathrm{-Mod}$ and $S\mathrm{-Mod}$ induces an isomorphism between $R\mathrm{-pr}$ and
$S\mathrm{-pr}$, the corresponding lattices of preradicals. In \cite{FernandezMagana2015} this result is generalized to the fact that every adjoint pair between two categories of modules induces a Galois connection between the corresponding lattices of preradicals. In this paper, we step further in the generalization, proving in Section 4 our main result, that every adjoint pair $(F,G)$  between locally small abelian categories $\mathcal{A}$ and $\mathcal{B}$ induces a Galois connection $(\varphi,\psi)$ between the corresponding collections $Pr(\mathcal{A})$ and $Pr(\mathcal{B})$ of preradicals (Theorem \ref{thmgaloisconexion1}). As a consequence, we have the result described at the beginning of this introduction in the context of abelian categories: every equivalence induces an order isomorphism between the collections of preradicals. To achieve the main theorem and some properties of this Galois connection, we define in Section 3 both operations of preradicals, product and coproduct, in abelian categories, and we introduce idempotent preradicals and radicals. We also study preradicals in the opposite category $\mathcal{A}^{op}$ and define the duality assignment $\Delta$ between both collections of preradicals, which is an order anti-isomorphism that exchanges both operations and both types of preradicals, idempotent and radicals. This assignment will be very useful to prove dual statements along the paper. In Section 5 we study this Galois connection from the point of view of the opposite categories. In Section 6 we define and study alpha preradicals and omega preradicals, which are already known in the context of categories of modules, now in the context of bicomplete abelian categories. And we prove that, in this context, the Galois connection induced by an adjoint pair preserves both types of preradicals. In Section 7 we prove that if $\mathcal{A}$ is bicomplete and locally small $Pr(\mathcal{A})$ is a complete lattice. This is achieved by constructing the join and the meet of an arbitrary collection of preradicals. It is proven that the collection of idempotent preradicals is closed under arbitrary joins, and the collection of radicals is closed under arbitrary meets. Every preradical can be described as a join of alpha preradicals, and every radical can be described as a meet of radicals. Also idempotent preradicals and radicals are characterized in terms of alpha and omega preradicals, respectively. This allows us to prove another important result: in the Galois connection $(\varphi,\psi)$ induced by an adjoint pair, $\varphi$ preserves idempotent preradicals and $\psi$ preserves radicals (Theorem \ref{galradidem}). Finally, in Section 8 we give two examples, where the lattice of idempotent radicals is described completely. In the first one we consider the category of left $\Lambda$-modules, where $\Lambda$ is a path algebra. In the second one, it is considered the category of torsion abelian groups.

 \section[preliminaries]{Preliminaries}\label{sec2}

\subsection[]{Galois connections}

In this section we present some basic concepts and results about Galois connections on posets. These can be extended to partially ordered classes and conglomerates. For more details, see \cite{Erne}.

Let $(P,\leq)$ and $(Q,\preceq)$ be posets. A Galois connection between posets $(P,\leq)$ and $(Q,\preceq)$ consists of two maps $f:P\longrightarrow Q$ and $g:Q\longrightarrow P$ such that for all $p\in P$ and $q\in Q$:
$$p \leq  g(q) \Longleftrightarrow  f(p)\preceq q.$$

A Galois connection will be denoted by the quadruple $\langle P, f, g,  Q\rangle$, or if it is clear, by the pair of maps $\langle f, g\rangle.$ The maps $f$ and $g$ are called, respectively, the $\textbf{coadjoint part}$ and the $\textbf{adjoint part}$. The following properties hold:
\begin{enumerate}
\item $f$ and $g$ preserve order.

\item $g\circ f$ is inflationary, i.e., for each $p\in P$, $(g\circ f)(p) \geq  p$.

\item $f\circ g$ is deflationary, i.e., for each $q\in Q$, $(f\circ g)(q)\preceq q$.

\item $f$ and $g$ are quasi-inverse, that is, $f\circ g\circ f=f$ and $g\circ f\circ g= g.$

\item If $P$ and $Q$ are complete lattices with least elements $0_{P}$ and $0_{Q}$, respectively, and greatest elements $1_{P}$ and $1_{Q}$, respectively, then $f$ preserves arbitrary joins, $g$ preserves
arbitrary meets, $f(0_{P})=0_{Q}$ and $g(1_{Q})=1_{P}$.

\item For each $p\in P$ , $(g\circ f)(p)=p \Longleftrightarrow  p\in  g(Q)$, thus $(g\circ f)(P)=g(Q)$.

\item  For each $q\in Q$, $(f\circ g)(q)=q \Longleftrightarrow q\in f(P)$, thus $(f\circ g)(Q)=f(P)$.

\end{enumerate}

In fact, posets $P$ and $Q$, and maps $f:P\longrightarrow Q$ and $g:Q\longrightarrow P$ constitute a Galois connection if, and only if, properties $1-3$ hold.
If $(P,\leq )$ is a poset, a map $h:P\longrightarrow P$ is a $\textbf{closure}$ ($\textbf{interior}$) $\textbf{operator}$ if
it preserves order, it is idempotent and inflationary (deflationary). A closure (interior) system of $P$ is a subset $Q$ of $P$ such that for each $p\in P$ there is a least (greatest) element $q\in Q$
such that $p\leq q$ ($p\geq q$). The element $q$ is called the closure (interior) of $p$. Any closure
(interior) system corresponds naturally to a closure (interior) operator, and conversely. If $h$ is a closure (interior) operator, the elements of $\mathrm{Im}(h)$ are called the $\textbf{closed}$ ($\textbf{open}$) elements of $P$. Thus properties $1-3$ imply that if $\langle P, f, g,  Q\rangle$ is a Galois connection then $g\circ f$ is a closure operator in $P$ and $f\circ g$ is a interior operator in $Q$. Therefore $g(Q)$ is the set of closed elements in $P$ and $f(P)$ is the set of open elements in $Q$. This Galois connection induces a poset isomorphism between $g(Q)$ and $f(P).$

\subsection[]{Subobjects, functors and adjoint pairs}
Let  $\mathcal{C}$ be a category and $X\in \mathcal{C}$,  the $\textbf{monomorphism category}$ associated to $X$, is the category $M(X)$ whose objects are monomorphisms in $\mathcal{C}$ with codomain $X$. Given two objects $u_{X'}:X'\rightarrow X$ and $u_{X''}:X''\rightarrow X$ in $M(X)$,  a morphism from $u_{X'}$ to $u_{X''}$ is a morphism $\alpha:X'\rightarrow X''$ in $\mathcal{C}$, such that $u_{X''}\circ \alpha=u_{X'}$. In this case we will write $u_{X'}\leq u_{X''}$. This defines a preorder on $M(X)$, which induces an equivalence relation $\simeq$, i.e., $u_{X'}\simeq u_{X''}$ if and only if $u_{X'}\leq u_{X''}$ and $u_{X''}\leq u_{X'}$. A $\textbf{subobject}$ of $X$ is an equivalence class of a monomorphism $u_{X'}:X'\rightarrow X$. A category $\mathcal{C}$ is said to be $\textbf{locally small}$ if the class of subobjects of any given object is actually a set. The preorder $\leq$ on $M(X)$ induces a partial order, which we denote with the same symbol, on the class of subobjects of $X$. For simplicity, we will refer to a subobject of $X$ by writing $u_{X'}:X'\longrightarrow X$,  understanding that we are in fact considering the equivalence class of $u_{X'}$. Therefore, we write $u_{X'}\leq u_{X''}$ when the corresponding subobjects are related by the partial order induced by $\leq$. Dually, we can define the epimorphism category associated to an object of $\mathcal{C}$, and a quotient object as an equivalence class of an epimorphism.\\

We recall that a category $\mathcal{C}$ is said to be $\textbf{small}$ if the class of objects is a set. Let $\mathcal{A}$ be a small category and $\mathcal{B}$ an arbitrary category. We denote by $\mathrm{Fun}(\mathcal{A},\mathcal{B})$ the collection of all covariant functors $F:\mathcal{A}\longrightarrow \mathcal{B}$. It is well-known that in this case $\mathrm{Fun}(\mathcal{A},\mathcal{B})$ is a category whose objects are the  covariant functors $F:\mathcal{A}\longrightarrow \mathcal{B}$  and a morphism  between two functors $F,G:\mathcal{A}\longrightarrow \mathcal{B}$ is just a natural transformation $\eta:F\longrightarrow G.$ If $\mathcal{A}$ is not a small category, then $\mathrm{Fun}(\mathcal{A},\mathcal{B})$ is an \textit{extended} category, in the sense that for any pair of such functors, the collection of all natural transformations between them is not necessarily a set. Throughout this paper when we refer to a category we mean a locally small category (not necessarily small).\\

We also recall two type of natural transformations which result from the composition between functors and natural transformations. In fact, they are result of the $\textbf{star composition}$ of any natural transformation and the natural identity on any functor, which for simplicity is denoted with the same symbol as the functor. Let $F,G:\mathcal{A}\longrightarrow \mathcal{B}$ be functors and $\eta:F\longrightarrow G$ a natural transformation. 
\begin{enumerate}
\item If $E:\mathcal{C}\longrightarrow \mathcal{A}$ is a funtor, we have a natural transformation $\eta E:F\circ E\longrightarrow G\circ E$ defined as follows: $[\eta E]_{C}:=\eta_{E(C)}:F(E(C))\longrightarrow G(E(C))$ for all $C\in \mathcal{C}$.

\item If $H:\mathcal{B}\longrightarrow \mathcal{C}$ is a functor, we have a natural transformation $H\eta:H\circ F\longrightarrow H\circ G$ defined as follows: $[H\eta]_{A}:=H(\eta_{A}):H(F(A))\longrightarrow H(G(A))$ for all $A\in \mathcal{A}$.
\end{enumerate}
For more details about these natural transformations see \cite[II.5]{MacLane1978} in p. 42.\\
Consider functors $F:\mathcal{A}\rightarrow\mathcal{B}$ and $G:\mathcal{B}\rightarrow\mathcal{A}$. We say that $F$ is $\textbf{left adjoint}$ to 
$G$ or that $G$ is $\textbf{right adjoint}$ to $F$, and  that $(F,G)$ is an adjoint pair, if there is a natural isomorphism
$$\varphi =\Big\{\varphi_{A,B}:\mathrm{Hom}_{\mathcal{B}}(FA,B)\rightarrow \mathrm{Hom}_{\mathcal {A}}(A,GB)\Big\}_{A\in\mathcal{A},B\in\mathcal{B}}$$
between the bifunctors $\mathrm{Hom}_{\mathcal{B}}(F(-),-)$ and $\mathrm{Hom}_{\mathcal{A}}(-,G(-))$.\\
Being $F$ left adjoint to $G$ is equivalent to the existence of a natural transformation $\eta:1_{\mathcal{A}}\longrightarrow GF$, called the $\textbf{unit}$, and a natural transformation $\varepsilon: FG\longrightarrow 1_{\mathcal{B}}$, called the $\textbf{counit}$, such that
 the so called $\textbf{triangular identities}$ hold:
 $$(G\varepsilon)\circ (\eta G)=1_{G},\quad \quad (\varepsilon F)\circ (F\eta)=1_{F}$$

\section[product and coproduct of preradicals]{Product and coproduct of preradicals in abelian categories}
%{\color{blue}[Idea de Rogelio]}
Throughout this paper we will work with abelian categories. For more details the reader can see \cite{MitBook} or \cite[Chapter IV]{Stentrom}. Let $\mathcal{A}$ be an abelian category. Then $\mathrm{Fun}(\mathcal{A},\mathcal{A})$ is an (extended) abelian category. A $\textbf{preradical}$ of $\mathcal{A}$ is a subobject of the identity functor $1_{\mathcal{A}}:\mathcal{A}\longrightarrow \mathcal{A}$ in $\mathrm{Fun}(\mathcal{A},\mathcal{A})$, i.e., it is the equivalence class of a monomorphism $\xymatrix{\tau: T\  \ar@{^(->}[r] & 1_{\mathcal{A}}}$ (which is thus a natural transformation) in $\mathrm{Fun}(\mathcal{A},\mathcal{A})$. Notice that in this case,  $\xymatrix{\tau_{A}: T(A)\ \ar@{^(->}[r] & A}$ is a monomorphism in $\mathcal{A}$ for all $A\in \mathcal{A}$. For simplicity, as it is explained in the previous section, when we refer in this paper to a preradical of $\mathcal{A}$ we will just write $\xymatrix{\tau: T\  \ar@{^(->}[r] & 1_{\mathcal{A}}}$ meaning the equivalence class of $\tau.$ Sometimes we may refer to any representative of a preradical $\tau$. Given an abelian category $\mathcal{A}$, we will denote by $Pr(\mathcal{A})$ the collection (not necessatily a set) of all preradicals of $\mathcal{A}$.\\

Recall (see \cite{FernandezRaggi}) that a preradical $T$ in a ring $R$ is a subfunctor of the identity functor on the category $R$-Mod of left $R$-modules. That is, we have the inclusion $\xymatrix{\iota: T\  \ar@{^(->}[r] & 1_{R\text{-}\mathrm{Mod}}}$ as a natural transformation. The collection of all preradicals in $R$ is denoted by $R$-pr. A reasonable question is if both collections of preradicals, $R$-pr and $Pr(R\text{-}\mathrm{Mod})$ coincide, considering the abelian category $R$-Mod.  \\

\begin{proposition}\label{preradiso}
    For any ring $R$, the ordered collections $R$-pr and $Pr(R\text{-}\mathrm{Mod})$ are isomorphic.
\end{proposition}

\begin{proof}
    Let $T,S\in R$-pr, and consider the natural inclusions $\xymatrix{\iota_T: T\  \ar@{^(->}[r] & 1_{R\text{-}\mathrm{Mod}}}$ and $\xymatrix{\iota_S: S\  \ar@{^(->}[r] & 1_{R\text{-}\mathrm{Mod}}}$. These inclusions represent subobjects, i.e., preradicals, in the abelian category $R$-Mod, which we denote in the same manner, $\iota_T, \iota_S\in Pr(R\text{-}\mathrm{Mod})$. Suppose that $\iota_S\leq\iota_T$. Then there exists a natural transformation $f:S\to T$ such that $\iota_T\circ f=\iota_S$. Therefore, for all $M\in R$-Mod, if $x\in S(M)$ then $x=f_M(x)\in T(M)$. Thus, $S(M)\leq T(M)$, and we have $S\leq T$. We conclude that if $\iota_S, \iota_T$ represent the same preradical in $Pr(R\text{-}\mathrm{Mod})$ then $S=T$.\\

    On the other hand, if $\xymatrix{\sigma: S\  \ar@{^(->}[r] & 1_{R\text{-}\mathrm{Mod}}}$ is any preradical in $Pr(R\text{-}\mathrm{Mod})$, we define the preradical $\sigma': \text{Im }\sigma\to 1_{R\text{-}\mathrm{Mod}}$ as follows: for any $M\in R\text{-Mod}$, $\sigma'_M:S'M\hookrightarrow M$ is the inclusion, where $S'M:=\text{Im }\sigma_M$. Then $S'$ is a preradical in $R$ and for each $M\in R\text{-Mod}$ we have the epimorphism $\eta_M:SM\to S'M$ which is the correstriction of $\sigma_M$ to its image, such that $\sigma_M=\sigma'_M\circ\eta_M$. Since $\sigma_M$ is a monomorphism then $\eta_M$ is also a monomorphism, and thus an isomorphism. Therefore $\sigma\cong\sigma'$, i.e., both represent the same preradical in $Pr(R\text{-}\mathrm{Mod})$.\\

    Therefore we can define an assignment $\varphi:R\text{-}pr\rightarrow Pr(R\text{-}\mathrm{Mod})$ such that $\varphi: S\mapsto \iota_S$, which, by the previous paragraphs, preserves order and is bijective. 
    
\end{proof}

As it is known, in any abelian category $\mathcal{A}$, each morphism $f:A\longrightarrow B$ has a kernel and a cokernel, which are unique, up to isomorphism, satisfying the corresponding property. In fact, we may refer to the $\textbf{subobject}$ $\mathrm{Ker}(f)$ of $A$, which is the equivalence class, and consists of all kernels of $f$. We may denote one of its representatives as a monomorphism $\mathrm{Kr}(f)\longrightarrow A$. We may also refer to the $\textbf{quotient object}$ $\mathrm{Coker}(f)$ of $B$, which consists of all cokernels of $f$. Similarly, we may denote one of its representatives as an epimorphism $B\longrightarrow\mathrm{Ckr}(f)$. Here $\mathrm{Kr}(f)$ and $\mathrm{Ckr}(f)$ denote objects of $\mathcal{A}$. Conversely, recall that any monomorphism of an abelian category $\mathcal{A}$ is a kernel, namely, a kernel of any of its cokernels, so if $u$ denotes a subobject then we may write $u=\mathrm{Ker}(\mathrm{Coker}(u))$. Similarly, if $v$ is any quotient object, we may write $v=\mathrm{Coker}(\mathrm{Ker}(v))$. In particular, being a subobject of the identity functor, any preradical $\sigma$ is such that $\sigma=\mathrm{Ker}(\mathrm{Coker}(\sigma))$. In fact, $\sigma$ determines a $\textbf{SQ-sequence}$ (for subobject-quotient object sequence) in $\mathcal{A}$:

\centerline{$\xymatrix{ S\ar@{^(->}[r]^{\sigma} & 1_{\mathcal{A}}\ar@{->>}[r]^{\sigma^*} & S^*}$}

\noindent which we may define as such that (1) $\sigma$ is a subobject, (2) $\sigma^*$ is a quotient object, (3) $\sigma^*= \mathrm{Coker}(\sigma)$, and (4) $\sigma=\mathrm{Ker}(\sigma^*)$. Of course we may also consider this as a short exact sequence in $\mathrm{Fun}(\mathcal{A},\mathcal{A})$, interpreting each one of the morphisms as representative of its equivalence class. Notice that, in this exact sequence, $\sigma_1$ and $\sigma_2$ are both representatives of the same subobject if, and only if, $\sigma^*_1$ and $\sigma^*_2$ are both representatives of the same quotient object.\\

Now we define both operations in $Pr(\mathcal{A})$, as well as idempotent preradicals and radicals. First notice that for any two preradicals $\xymatrix{\tau:T\ar@{^(->}[r] & 1_{\mathcal{A}}}$ and $\xymatrix{\sigma:S\ar@{^(->}[r] & 1_{\mathcal{A}}}$ of $\mathcal{A}$, we have the following commutative diagram in $\mathrm{Fun}(\mathcal{A},\mathcal{A})$:

\begin{equation}\label{SQdiagram}
\begin{aligned}
\xymatrix{ & 0\ar[d] & 0\ar[d] & 0\ar[d]\\
& TS\ar[d]_{\tau S}\ar[r]^{T\sigma} & T\ar[d]^{\tau}\ar[r]^{T\sigma^*}  & TS^*\ar[d]^{\tau S^*} &\\
0\ar[r] & S\ar[r]^{\sigma}\ar[d]_{\tau^* S} & 1_{\mathcal{A}}\ar[r]^{\sigma^*}\ar[d]^{\tau^*} &  S^*\ar[r]\ar[d]^{\tau^* S^*} & 0\\
& T^*S\ar[d]\ar[r]_{T^*\sigma} & T^*\ar[d]\ar[r]_{T^*\sigma^*} & T^*S^*\ar[d] &\\
 & 0 & 0 & 0}
 \end{aligned}
\end{equation}

The upper squares are commutative by the naturality of any representative of $\tau$; the lower squares commute by the naturality of any representative of $\tau^*$. The columns of the diagram are short exact sequences corresponding to the SQ-sequence determined by $\tau$. The middle row is the short exact sequence corresponding to the SQ-sequence determined by $\sigma$.
\vspace{.2in}

\begin{definition}\label{product}
Let  $\xymatrix{\tau: T\ar@{^(->}[r] & 1_{\mathcal{A}}}$ and $\xymatrix{\sigma: S\ar@{^(->}[r] & 1_{\mathcal{A}}}$ be preradicals. Considering the left upper square of diagram \eqref{SQdiagram}, we define the $\textbf{product}$ of $\tau$ and $\sigma$, denoted by $\tau\cdot \sigma$, as the equivalence class of $\sigma \circ \tau S=\tau\circ T\sigma$.
\end{definition}
\vspace{.2in}

Notice that this product is well defined, because if $\sigma, \sigma'$ are representatives of the same preradical, and $\tau, \tau'$ are representatives of another preradical, i.e., if $\sigma\cong\sigma'$ and $\tau\cong \tau'$, then $\sigma \circ \tau S \cong \sigma' \circ \tau' S $, i.e., this composition represents the same preradical $\tau\cdot \sigma$. In particular, for any preradical $\xymatrix{\tau: T\ar@{^(->}[r] & 1_{\mathcal{A}}}$, the commutative diagram involving a representative of $\tau$ that defines the product $\tau\cdot \tau$ is:

\begin{equation}\label{idem}
\begin{aligned}
\xymatrix{T^{2}\ar[r]^{\tau T}\ar[d]_{T\tau} & T\ar[d]^{\tau}\\
T\ar[r]_{\tau} & 1_{\mathcal{A}}}
\end{aligned}
\end{equation}

Notice that, since the representative of $\tau$ is a monomorphism, we have that $T\tau=\tau T.$
\vspace{.2in}

\begin{definition}\label{defiidempoten}
Let $\xymatrix{\tau: T\ar@{^(->}[r] & 1_{\mathcal{A}}}$ be a preradical. We say that $\tau$ is $\textbf{idempotent}$ if $\tau\cdot \tau = \tau$.
\end{definition}
\vspace{.2in}

\begin{remark}
For any preradical $\xymatrix{\tau: T\ar@{^(->}[r] & 1_{\mathcal{A}}}$, $\tau$ is idempotent if, and only if, $T\tau=\tau T:T^{2}\longrightarrow T$ is an isomorphism (see diagram \eqref{idem}). 
\end{remark}
\vspace{.2in}
There is some similarity between Definition \ref{product} and the definition of product of preradicals on modules. in the same manner, considering the definition of coproduct of preradicals on modules, we give the following definition.
\vspace{.2in}

\begin{definition}\label{coproduct}
Let $\xymatrix{\tau:T\ar@{^(->}[r] & 1_{\mathcal{A}}}$ and $\xymatrix{\sigma:S\ar@{^(->}[r] & 1_{\mathcal{A}}}$ be preradicals. We define the $\textbf{coproduct}$ of $\sigma$ and $\tau$, denoted by $(\sigma : \tau)$, as the preradical whose representative is part of the pullback diagram in $\mathrm{Fun}(\mathcal{A},\mathcal{A})$:

\centerline{$\xymatrix{(S:T)\ar@{^(->}[r]^(.6){  (\sigma:\tau)}\ar[d]^{} & 1_{\mathcal{A}}\ar@{->>}[d]^{\sigma^*} & \\
TS^*\ar@{^(->}[r]^{\tau S^*} & S^*}$}
\end{definition}
\vspace{.2in}

Again, this coproduct is well defined, because if $\sigma_1\cong\sigma_2$ and $\tau_1\cong\tau_2$ then the monomorphism that completes the pullback represents the same preradical. We shall give an alternative definition, using the SQ-sequence determined by any preradical in an abelian category. We recall first the following categorical property.
\vspace{.2in}

\begin{lemma}\label{lemkerpull}
Consider the following diagram in a category with zero object
$$\xymatrix{A'\ar[r]^{\beta} & A\ar[d]^{\alpha_{1}} & \\
B'\ar[r]^{\alpha_{2}} & B\ar[r]^{\alpha_{3}} & B''}$$
where $\alpha_{2}$ is a kernel of $\alpha_{3}$. Then the above diagram can be completed to a pullback if, and only if, $\beta$ is a kernel of $\alpha_{3}\alpha_{1}$. 
\end{lemma}
\begin{proof}
See \cite[Proposition 13.1, p.15]{MitBook}.
\end{proof}
\vspace{.2in}

\begin{proposition}\label{otradescop}
Let $\xymatrix{\tau:T\ar@{^(->}[r] & 1_{\mathcal{A}}}$ and $\xymatrix{\sigma:S\ar@{^(->}[r] & 1_{\mathcal{A}}}$ be preradicals of $\mathcal{A}$. Considering the lower right square of diagram \eqref{SQdiagram}, we have: $$(\sigma : \tau)=\mathrm{Ker}(\xymatrix{1_{\mathcal{A}}\ar[r]^{\sigma^*} & S^*\ar[r]^{\tau^* S^*} & T^*S^*)}$$

\end{proposition}

\begin{proof}
It follows applying Lemma \ref{lemkerpull} to Definition \ref{coproduct}.
\end{proof}

In particular, if $\tau=\sigma$, the lower right square of diagram \eqref{SQdiagram} looks like:

\begin{equation}\label{rad}
\begin{aligned}
\xymatrix{
0\ar[r] & S\ar[r]^{\sigma} & 1_{\mathcal{A}}\ar[r]^{\sigma^*}\ar[d]^{\sigma^*} &  S^*\ar[r]\ar[d]^{\sigma^* S^*} & 0\\
& & S^*\ar[r]^{S^*\sigma^*} & S^*S^* &}
\end{aligned}
\end{equation}

Notice that, since the representative of $\sigma^*$ is an epimorphism, then $S^*\sigma^*=\sigma^*S^*$.
 \vspace{.2in}
\begin{definition}
Let  $\xymatrix{\sigma:S\ar@{^(->}[r] & 1_{\mathcal{A}}}$ be a preradical of $\mathcal{A}$. We say that $\sigma$ is a $\textbf{radical}$ if  $(\sigma:\sigma)=\sigma$.
 \end{definition}
 \vspace{.2in}
 
\begin{remark}\label{isirad}
For any preradical $\xymatrix{\sigma: S\ar@{^(->}[r] & 1_{\mathcal{A}}}$, due to Proposition \ref{otradescop}, $\sigma$ is a radical if, and only if, $\sigma=\mathrm{Ker}(\xymatrix{1_{\mathcal{A}}\ar[r]^{\sigma^*} & S^*\ar[r]^{\sigma^* S^*} & S^*S^*)}$, which is equivalent to the condition that $\sigma^*$ and $S^*\sigma^*\circ\sigma^*$, being both epimorphisms, represent the same quotient object, i.e.,
considering  diagram \eqref{rad}, $S^*\sigma^*=\sigma^*S^*:S^*\longrightarrow S^*S^*$ is an isomorphism. 
\end{remark}

For the abelian category $\mathcal{A}$, let $\mathcal{A}^{op}$ be the corresponding opposite category. We can define an assignment $\Delta: Pr(\mathcal{A})\rightarrow Pr(\mathcal{A}^{op})$, such that $\Delta(\tau):=(\tau^*)^{op}$. Notice that indeed, if $\tau:T\to 1_{\mathcal{A}}$ then $\tau^*:1_{\mathcal{A}}\to T^*$ and $(\tau^*)^{op}:T^{op}\to 1_{\mathcal{A}^{op}}$. Here $T^{op}:\mathcal{A}^{op}\to \mathcal{A}^{op}$ is the functor such that for any object  $A$ in $\mathcal{A}^{op}$ we have $T^{op}(A)=T(A)$ and for any morphism $f^{op}$ in $\mathcal{A}^{op}$ we have
$T^{op}(f^{op})=T(f)^{op}$. The assignment $\Delta$ is a one-to-one correspondence, with inverse assignment $\Delta': Pr(\mathcal{A}^{op})\rightarrow Pr(\mathcal{A})$ defined similarly. Notice also that, for any $\tau,\sigma\in Pr(\mathcal{A})$ we have $\tau\leq\sigma$ if, and only if $\tau^*\leq\sigma^*$ if, and only if $(\sigma^*)^{op}\leq(\tau^*)^{op}$. Therefore $\Delta$ is an anti-isomorphism of ordered structures. We call $\Delta$ the $\textbf{duality assignment}$ related to the category $\mathcal{A}$, and when necessary we will denote it by $\Delta_{\mathcal A}$. Furthermore, this assignment relates both operations of preradicals in $Pr(\mathcal{A})$ and $Pr(\mathcal{A}^{op})$.  
\vspace{.2in}

\begin{proposition}\label{delta}
Let $\mathcal{A}$ be an abelian category and let $\Delta: Pr(\mathcal{A})\rightarrow Pr(\mathcal{A}^{op})$ be the assignment defined above. For any $\tau,\sigma\in Pr(\mathcal{A})$ we have:
\begin{enumerate}
    \item [(a)] $\Delta(\sigma : \tau)=\Delta(\tau)\cdot\Delta(\sigma)$
    \item [(b)] $\Delta(\tau\cdot\sigma)=(\Delta(\sigma) : \Delta(\tau))$
\end{enumerate}
\end{proposition}
\begin{proof}
    As we said above, for any two preradicals $\xymatrix{\tau:T\ \ar@{^(->}[r] & 1_{\mathcal{A}}}$ and $\xymatrix{\sigma:S\ \ar@{^(->}[r] & 1_{\mathcal{A}}}$ of $\mathcal{A}$, we have diagram (\ref{SQdiagram}) in $\mathrm{Fun}(\mathcal{A},\mathcal{A})$. The corresponding diagram in $\mathrm{Fun}(\mathcal{A}^{op},\mathcal{A}^{op})$ is:
\begin{equation}\label{SQopdiagram}
\begin{aligned}
\xymatrix{ & 0\ar[d] & 0\ar[d] & 0\ar[d]\\
& T^*S^*\ \ \ar[d]_{(\tau^* S^*)^{op}}\ar[r]^{(T^*\sigma^*)^{op}} & T^*\ar[d]^{\Delta(\tau)}\ar[r]^{(T^*\sigma)^{op}}  & T^*S\ar[d]^{(\tau^* S)^{op}} &\\
0\ar[r] & S^*\ar[r]^{\Delta(\sigma)}\ar[d]_{(\tau S^*)^{op}} & 1_{\mathcal{A}^{op}}\ar[r]^{\sigma^{op}}\ar[d]^{\tau^{op}} &  S\ar[r]\ar[d]^{(\tau S)^{op}} & 0\\
& TS^*\ar[d]\ar[r]_{(T\sigma^*)^{op}} & T\ar[d]\ar[r]_{(T\sigma)^{op}} & TS\ar[d] &\\
 & 0 & 0 & 0}
 \end{aligned}
\end{equation}

By Definition \ref{product} we have: 
\begin{align*}
\Delta(\tau)\Delta(\sigma)=\Delta(\sigma)\circ(\tau^*S^*)^{op}&=(\sigma^*\circ(\tau^*S^*))^{op}=\Delta(\sigma :\tau) 
\end{align*}
The last equality holds by Proposition \ref{otradescop} and by definition of $\Delta$.\\

On the other hand, by Proposition \ref{otradescop} we have:
\begin{align*}
    (\Delta(\sigma) : \Delta(\tau))&=\text{Ker}(\Delta(\sigma) : \Delta(\tau))^*\\
    &=\text{Ker}( (T\sigma)^{op}\circ\tau^{op})\\
 &=\text{Ker}(\tau\cdot\sigma)^{op}=((\tau\cdot\sigma)^*)^{op}\\
 &=\Delta(\tau\cdot\sigma)
\end{align*}
The third equality holds by Definition \ref{product}, and the last equality by definition of $\Delta$. 
\end{proof}
\vspace{.2in}

\begin{corollary}\label{radidem}
Let $\mathcal{A}$ be an abelian category and let $\tau\in Pr(\mathcal{A})$. Then:
\begin{enumerate}
    \item[(a)] $\tau$ is a radical if, and only if, $\Delta(\tau)$ is idempotent in $Pr(\mathcal{A}^{op})$
    \item[(b)] $\tau$ is idempotent if, and only if, $\Delta(\tau)$ is a radical in $Pr(\mathcal{A}^{op})$
\end{enumerate}
\end{corollary} \qed
\vspace{.2in}

For any abelian category $\mathcal{A}$, we will denote by $Idem(\mathcal{A})$ the collection of all idempotent preradicals of $\mathcal{A}$, by $Rad(\mathcal{A})$, the collection of all radicals  of $\mathcal{A}$, and by $Radidem(\mathcal{A})$ the collection of all idempotent radicals of $\mathcal{A}$. \\

\begin{corollary}\label{antisos}
For any abelian category  $\mathcal{A}$, the anti-isomorphism $\Delta$ induces poset anti-isomorphisms:  
$$\Delta': Idem(\mathcal{A})\longrightarrow Rad(\mathcal{A}^{op})$$
$$\Delta'': Rad(\mathcal{A})\longrightarrow Idem(\mathcal{A}^{op})$$
$$\Delta''': Radidem(\mathcal{A})\longrightarrow Radidem(\mathcal{A}^{op})$$
\end{corollary}\qed

\section[Galois connections]{Galois connections between preradicals in abelian categories}

Recall that, in every abelian category, each morphism $f:A\longrightarrow B$ has a factorization
$f=up$ where $p$ is an epimorphism and $u$ is a monomorphism. In fact, in any such a factorization, $p$ is a representative of the quotient object $\mathrm{Coker}(\mathrm{Ker}(f))$ and  $u$ is a representative of the subobject $\mathrm{Ker}(\mathrm{Coker}(f))$, which is by definition the $\textbf{image}$ of $f$, denoted by $\mathrm{Im}(f)$. The following result will be useful and it is easy to prove. Notice that in part $(a)$ we are denoting by $\leq$ the partial order on the subobjects of $C$, and in part $(b)$ we use the same symbol to denote the subobjects of $A$.\\

\begin{lemma}\label{lemimker}
Let $\mathcal{C}$ be an abelian category.  Let $f:A\longrightarrow B$ and $g:B\longrightarrow C$ morphisms in $\mathcal{C}$. Then the following properties hold:
\begin{enumerate}
\item [(a)] $\mathrm{Im}(gf)\leq \mathrm{Im}(g)$.
\item [(b)] $\mathrm{Ker}(f)\leq \mathrm{Ker}(gf)$.
\end{enumerate}
\end{lemma}
\qed
\vspace{.2in}

In what follows, we consider two abelian categories $\mathcal{A}$ and $\mathcal{B}$. Then we have the abelian categories  $\mathrm{Fun}(\mathcal{A},\mathcal{A})$  and  $\mathrm{Fun}(\mathcal{B},\mathcal{B})$. As we said before, there is a partial order
$\leq$ on the collection $Pr(\mathcal{A})$ of all preradicals of $\mathcal{A}$. We will denote with the same symbol the partial order on $Pr(\mathcal{B})$.
\vspace{.2in}

\begin{theorem}\label{thmgaloisconexion1}
Each adjoint pair $(F,G):\mathcal{A}\longrightarrow \mathcal{B}$ between abelian categories induces a Galois connection $(\varphi,\psi):Pr(\mathcal{A})\longrightarrow Pr(\mathcal{B})$.
\end{theorem}
\begin{proof}
Consider the unit $\eta:1_{\mathcal{A}}\longrightarrow GF$, and counit $\varepsilon: FG\longrightarrow 1_{\mathcal{B}}$ of the adjoint pair $(F,G)$. 
Let $\xymatrix{\tau: T\ar@{^(->}[r] & 1_{\mathcal{A}}}$ be in $Pr(\mathcal{A})$. We define the preradical:

\centerline{$\varphi(\tau):=\mathrm{Im}\Big(\xymatrix{FTG\ar[r]^{F\tau G} & FG\ar[r]^{\epsilon} & 1_{\mathcal{B}}}\Big)$}
\noindent where for any representative of the preradical $\tau$, $F\tau G:FTG\longrightarrow FG$ is the natural transformation given by 
$(F\tau G)_{B}:=F(\tau_{G(B)})$ for each $B\in \mathcal{A}$.\\

Now we show that $\varphi$ is well defined. Let $\tau_1,\tau_2$ be two representatives of $\tau$, i.e., $\tau_1\simeq\tau_2$. In particular, $\tau_{1}\leq \tau_{2}$, so there exists $\mu:T_{1}\longrightarrow T_{2}$ such that
$\tau_{1}=\tau_{2}\mu$. By applying the functors $F$ and $G$, respectively, to the left side and to the right side of the latter equality, we get the following commutative diagram:

\centerline{$
\xymatrix{FT_{1}G\ar[dr]^{F\tau_{1}G}\ar[d]_{F\mu G} & \\
F T_{2}G\ar[r]_{F\tau_{2}G} &  FG\ar[r]^{\epsilon} & 1_{\mathcal{B}} }$}

That is, $\epsilon\circ (F\tau_{1}G)=\Big(\epsilon\circ (F\tau_{2}G)\Big)\circ (F\mu G)$.
By Lemma \ref{lemimker}(a), we have $\varphi(\tau_{1})\leq \varphi(\tau_{2})$. Similarly, $\tau_2\leq\tau_1$ implies that $\varphi(\tau_2)\leq\varphi(\tau_1)$. Thus, $\tau_{1}\simeq \tau_{2}$ implies that $\varphi(\tau_{1})\simeq\varphi(\tau_{2})$. This also shows that $\varphi$ preserves order.\\

Notice that we have the following commutative diagram in $\mathrm{Fun}(\mathcal{B},\mathcal{B})$:
\begin{equation}\label{epimonophi}
\begin{aligned}
\xymatrix{ & S\ar@{^(->}[dr]^{\varphi(\tau)} &\\
FTG\ar[r]^{F\tau G}\ar@{->>}[ur] & FG\ar[r]^{\epsilon} & 1_{\mathcal{B}}}
\end{aligned}
\end{equation}

Now let $\xymatrix{\sigma: S\ar@{^(->}[r] & 1_{\mathcal{B}}}$ be in $Pr(\mathcal{B})$. We have, for each representative of $\sigma$, the natural transformation $G\sigma F:GSF\longrightarrow GF$ given by $(G\sigma F)_{A}:=G(\sigma_{F(A)})$ for all $A\in \mathcal{A}$. Let $\pi_\sigma$ be its cokernel in $\mathrm{Fun}(\mathcal{A},\mathcal{A})$:

\centerline{$\pi_{\sigma}:GF\longrightarrow \mathrm{Ckr}(G\sigma F)$}

And we define the preradical:

\centerline{$\psi(\sigma):=\mathrm{Ker}\Big(\xymatrix{1_{\mathcal{A}}\ar[r]^{\eta} & GF\ar[r]^(.4){\pi_{\sigma}} &  \mathrm{Ckr}(G\sigma F)}\Big)$}

\vspace{.2in}
Notice that it may also be described as:

\centerline{$\psi(\sigma):=\mathrm{Ker}\Big(\xymatrix{1_{\mathcal{A}}\ar[r]^{\eta} & GF\ar[r]^(.4){G\sigma^*F} &  GS^* F)}\Big)$}
\vspace{.2in}

To show that $\psi$ is well defined, let $\sigma_1,\sigma_2$ be two representatives of $\sigma$, i.e., $\sigma_1\simeq\sigma_2$. In particular, $\sigma_{1}\leq \sigma_{2}$, so there exists $\nu:S_{1}\longrightarrow S_{2}$ such that
$\sigma_{1}=\sigma_{2}\nu.$ Applying the functors $G$ and $F$, respectively, to the left side and to the right side of the latter equality, we get the following commutative diagram:

\centerline{
$\xymatrix{GS_{1}F\ar[dr]^{G\sigma_{1}F}\ar[d]_{G\nu F} & \\
GS_{2}F\ar[r]_{G\sigma_{2}F} &  GF}$}

Then, there exists $\gamma: \mathrm{Ckr}(G\sigma_{1}F)\longrightarrow \mathrm{Ckr}(G\sigma_{2}F)$ such that the following diagram commutes:

\centerline{$\xymatrix{GS_{1}F\ar[r]^{G\sigma_{1}F}\ar[d]_{G\nu F} &  GF\ar@{=}[d]\ar[r]^(.3){\pi_{\sigma_{1}}} & \mathrm{Ckr}(G\sigma_{1}F)\ar[r]\ar[d]^{\gamma} & 0 \\
GS_{2}F\ar[r]_{G\sigma_{2}F} &  GF\ar[r]^(.3){\pi_{\sigma_{2}}} & \mathrm{Ckr}(G\sigma_{2}F)\ar[r] & 0 }$}

Therefore, $\gamma\circ (\pi_{\sigma_{1}}\circ\eta)=\pi_{\sigma_{2}}\circ\eta$, and by Lemma \ref{lemimker}(b), we have $\psi(\sigma_{1})\leq \psi(\sigma_{2})$. Similarly, $\sigma_2\leq\sigma_1$ implies that $\psi(\sigma_2)\leq\psi(\sigma_1)$. Thus, $\sigma_{1}\simeq \sigma_{2}$ implies that $\psi(\sigma_{1})\simeq\psi(\sigma_{2})$. This also shows that $\psi$ preserves order.\\

Once defined the assignments $\varphi$ and $\psi$, we will prove that $(\varphi,\psi)$ is a Galois connection.

\begin{enumerate}

\item [(a)] We claim that  $\tau\leq \psi(\varphi(\tau))$ for each
$\xymatrix{\tau: T\ar@{^(->}[r] & 1_{\mathcal{A}}}$. By construction $\varphi(\tau)=\mathrm{Im}\Big(\xymatrix{FTG\ar[r]^{F\tau G} & FG\ar[r]^{\varepsilon} & 1_{\mathcal{B}}}\Big)$. Again, consider commutative diagram (\ref{epimonophi}). Then we have $G\varphi(\tau) F:GSF\longrightarrow GF$ and its cokernel $\pi_{\varphi(\tau)}:GF\longrightarrow \mathrm{Ckr}(G\varphi(\tau) F)$.
Then, by definition: 

\centerline{$\psi(\varphi(\tau))=\mathrm{Ker}\Big(\xymatrix{1_{\mathcal{A}}\ar[r]^{\eta} & GF\ar[r]^(.3){\pi_{\varphi(\tau)}} &  \mathrm{Ckr}(G\varphi(\tau) F)}\Big)$}

\vspace{.2in}
Now we claim that the following composition is zero:

\centerline{$\xymatrix{T\ar[r]^{\tau} & 1_{\mathcal{A}}\ar[r]^{\eta} & GF\ar[r]^(.3){\pi_{\varphi(\tau)}} &  \mathrm{Ckr}(G\varphi(\tau) F)}$}

\vspace{.2in}
Indeed, by applying $G$ to diagram (\ref{epimonophi}), and this diagram to the functor $F$, we get the commutative diagram:

\centerline{$\xymatrix{ &&  GSF\ar[drr]^{G\varphi(\tau)F} &\\
GFTGF\ar[rr]^{GF\tau GF}\ar[urr] &&  GFGF\ar[rr]^{G\varepsilon F} && GF\ar[r]^(.3){\pi_{\varphi(\tau)}} &  \mathrm{Ckr}(G\varphi(\tau) F)}$}

\vspace{.2in}
Since $\pi_{\varphi(\tau)}=\mathrm{Coker}(G\varphi(\tau)F)$ we have that $\pi_{\varphi(\tau)} \circ (G\varphi(\tau)F)=0$  and therefore $\pi_{\varphi(\tau)}\circ G\varepsilon F\circ GF\tau GF=0$.\\

We add this fact to the following commutative diagram:

\begin{equation}\label{psiphi}
\begin{aligned}
\xymatrix{T\ar[d]_{\tau}\ar[rr]^{T \eta}\ar@{}[drr]|{I}
& &  TGF\ar[d]_{\tau GF}\ar[rr]^{\eta TGF}\ar@{}[drr]|{II} &  & GFTGF\ar[d]_{GF\tau GF}\ar@/^3pc/[ddd]^{0} \\
1_{\mathcal{A}}\ar[rr]_{\eta} & & GF\ar[rr]_{\eta GF}\ar[drr]_{1} &  & GFGF\ar[d]_{G\varepsilon F}\ar@{}[dll]|(.3){III}\\
& & & & GF\ar[d]_{\pi_{\varphi(\tau)}}\\
& & & & \mathrm{Ckr}(G\varphi(\tau) F) }
\end{aligned}
\end{equation}

\noindent where square $I$ commutes because of the naturality of $\tau$, square $II$ commutes because of the naturality of $\eta$, and triangle $III$ commutes by the triangular identities of the adjoint pair $(F,G)$.
From this diagram, we have that $\pi_{\varphi(\tau)}\circ \eta \circ\tau=0.$ Since $\psi(\varphi(\tau))=\mathrm{Ker}(\pi_{\varphi(\tau)} \eta)$, we conclude that  $\tau\leq \psi(\varphi(\tau))$.\\

\item [(b)] Now we claim that $\varphi(\psi (\sigma))\leq \sigma$ for all $\xymatrix{\sigma: S\ar@{^(->}[r] & 1_{\mathcal{B}}}$. For any such $\sigma$ we have the natural transformation $G\sigma F:GSF\longrightarrow GF$ and we consider its cokernel $\pi_{\sigma}:GF\longrightarrow \mathrm{Ckr}(G\sigma F).$
By definition, we have:\\  \centerline{$\psi(\sigma)=\mathrm{Ker}\Big(\xymatrix{1_{\mathcal{A}}\ar[r]^{\eta} & GF\ar[r]^(.4){G\sigma^*F} &  GS^* F}\Big)$, and}

\centerline{$\varphi(\psi(\sigma))=\mathrm{Im}\Big(\xymatrix{FTG\ar[rr]^{F\psi(\sigma) G} & & FG\ar[r]^{\varepsilon} & 1_{\mathcal{B}}}\Big)$}
\vspace{.1in}
Therefore $G\sigma^* F\circ\eta\circ\psi(\sigma)=0$, and we have $FG\sigma^* FG\circ F\eta G\circ F\psi(\sigma) G=0$. We use this fact in the following diagram:

\begin{equation}\label{phipsi}
\begin{aligned}
\xymatrix{FTG\ar[d]^{F\psi(\sigma) G}\ar@/_3pc/[ddd]_{0} & &\\
FG\ar[d]^{F\eta G}\ar[drr]^{1} & &\\
FGFG\ar[d]^{FG \sigma ^*FG}\ar[rr]^{\varepsilon FG}\ar@{}[drr]|{II}\ar@{}[urr]|(.3){I} & & FG\ar[d]_{\sigma ^* F G}\ar[rr]^{\varepsilon}\ar@{}[drr]|{III}  & & 1_{\mathcal{B}}\ar[d]_{\sigma ^*}\\
FGS^*FG\ar[rr]_{\varepsilon S^*FG} & & S^*FG\ar[rr]_{S^*\varepsilon} & & S^*}
\end{aligned}
\end{equation}

\noindent where triangle $I$ commutes by the triangular identities, square $II$ commutes by the naturality of $\varepsilon$, and square $III$ commutes by the naturality of $\sigma$. 
Therefore $\sigma^*\circ\varepsilon\circ (F\psi(\sigma)G)=0$ and this implies that $\mathrm{Coker}\Big(\varepsilon\circ (F\psi(\sigma)G)\Big)\leq \sigma^*$. We conclude that $\varphi(\psi(\sigma))\leq \sigma$.

\end{enumerate}  
\end{proof}

We say that $(\varphi,\psi)$ is the Galois connection \textit{induced} by the adjoint pair $(F,G)$. In particular, we obtain the result presented in \cite{BJKN, BKN} for categories of modules. Namely, every equivalence between abelian categories induces an isomorphism between the corresponding collections of preradicals.
\vspace{.2in}

\begin{corollary}\label{equiviso}
Each equivalence $F:\mathcal{A}\longrightarrow\mathcal{B}$ between abelian categories induces an isomorphism $\varphi:Pr(\mathcal{A})\longrightarrow Pr(\mathcal{B})$.
\end{corollary}
\begin{proof}
 Let us suppose that $F:\mathcal{A}\longrightarrow\mathcal{B}$ is an equivalence. We can assume that there is a functor $G:\mathcal{B}\longrightarrow\mathcal{A}$ such that $(F,G)$ is an adjoint pair, and the unit  $\eta:1_{\mathcal{A}}\longrightarrow GF$, and counit $\varepsilon: FG\longrightarrow 1_{\mathcal{B}}$ are natural isomorphisms. By Theorem \ref{thmgaloisconexion1}, we can consider the Galois connection $(\varphi,\psi)$ induced by the pair $(F,G)$. Let $\tau\in Pr(\mathcal{A})$. Then $\tau\leq \psi(\varphi(\tau))$. On the other hand, notice that, since $\varepsilon$ is an isomorphism and $F$ is an equivalence, $\varepsilon\circ F\tau G$ is a monomorphism in $\mathcal{B}$; in fact it is equivalent to the preradical $\varphi(\tau)$. Therefore $G\varepsilon F\circ GF\tau GF$ is a monomorphism in $\mathcal{A}$, and in fact $\pi_{\varphi(\tau)}=\mathrm{Coker}(G\varepsilon F\circ GF\tau GF)$. Let $\xymatrix{\delta: D\ar@{^(->}[r] & 1_{\mathcal{B}}}$ be any preradical such that $\pi_{\varphi(\tau)}\circ \eta \circ\delta=0$. By the Factor Theorem, there exists $\gamma:D\longrightarrow GFTGF$ such that $ \eta \circ\delta=G\varepsilon F \circ GF\tau GF\circ\gamma$. Since $\eta$ is an isomorphism, then $\eta TGF\circ T\eta$ is an isomorphism. Let us denote $\beta=(\eta TGF\circ T\eta)^{-1}$. Then, by diagram (\ref{psiphi}), $\eta\circ\tau\circ\beta\circ\gamma=G\varepsilon F\circ GF\tau GF\circ(\eta TGF\circ T\eta)\circ\beta\circ\gamma=G\varepsilon F\circ GF\tau GF\circ\gamma=\eta\circ\delta$. Therefore $\tau\circ\beta\circ\gamma=\delta$, i.e., $\delta\leq\tau$. We conclude that $\psi(\varphi(\tau))=\tau$.\\
 
 Now let $\sigma\in Pr(\mathcal{A})$. Again by Theorem \ref{thmgaloisconexion1} we have $\varphi(\psi(\sigma))\leq \sigma$. Since $\eta$ is an isomorphism, and $F,G$ are equivalences, we have that $FG\sigma^*FG\circ F\eta G$ is an epimorphism in $\mathcal{B}$. By definition of $\psi(\sigma)$ and applying the equivalence $F$, we have $F\psi(\sigma) G=\mathrm{Ker}(FG\sigma^*FG\circ F\eta G)$. Let $\delta':1_{\mathcal{A}}\longrightarrow D'$ be any morphism in $\mathrm{Fun}(\mathcal{A}, \mathcal{A})$ such that $\delta'\circ \varepsilon\circ F\psi(\sigma)G=0$. By the Factor Theorem, there exists $\gamma':FGS^*FG\longrightarrow D'$ such that $\delta'\circ\varepsilon=\gamma'\circ FG\sigma^* FG\circ F\eta G$. Since $\varepsilon$ is an isomorphism, then $S^*\varepsilon\circ\varepsilon S^*FG$ is an isomorphism. Let $\beta'=(S^*\varepsilon\circ\varepsilon S^*FG)^{-1}$. Then, by diagram (\ref{phipsi}), $\gamma'\circ\beta'\circ\sigma^*\circ\varepsilon=\gamma'\circ\beta'\circ (S^*\varepsilon\circ\varepsilon S^*FG) \circ FG\sigma^* FG \circ F\eta G=\gamma'\circ FG\sigma^* FG\circ F\eta G =\delta'\circ\varepsilon$. Therefore $\gamma'\circ\beta'\circ\sigma^*=\delta'$, i.e. $\sigma^*\leq\delta'$. This means that $\sigma^*\leq\mathrm{Coker}\Big(\varepsilon\circ (F\psi(\sigma)G)\Big)$. We conclude that $\varphi(\psi(\sigma))=\sigma$. 
\end{proof}

\section{Galois connections in opposite categories}

Let us consider an adjoint pair $(F,G):\mathcal{A}\longrightarrow \mathcal{B}$ between abelian categories, with unit $\eta:1_{\mathcal{A}}\longrightarrow GF$, and counit $\varepsilon: FG\longrightarrow 1_{\mathcal{B}}$. Then we have also an adjoint pair $(G^{op},F^{op}):\mathcal{B}^{op}\longrightarrow \mathcal{A}^{op}$, where $F^{op}:\mathcal{A}^{op}\longrightarrow \mathcal{B}^{op}$ is the functor such that for any object  $A$ in $\mathcal{A}^{op}$ we have $F^{op}(A)=F(A)$ and for any morphism $f^{op}$ in $\mathcal{A}^{op}$ we have
$F^{op}(f^{op})=F(f)^{op}$. The functor $G^{op}$ is defined similarly. The unit of this adjoint pair is $\varepsilon':1_{\mathcal{B}^{op}}\longrightarrow F^{op}G^{op}$, and the counit is $\eta': G^{op}F^{op}\longrightarrow 1_{\mathcal{A}^{op}}$, where for each object $B$ in $\mathcal{B}^{op}$ we have that $\varepsilon'_B:B\to FG(B)$ is the opposite morphism, in $\mathcal{B}^{op}$, of $\varepsilon_B$, i.e., $\varepsilon'_B=\varepsilon_B^{op}$. The counit $\eta'$ is defined similarly. It is easy to verify that $\varepsilon'$ and $\eta'$ satisfy the triangle identities, from the corresponding triangle identities of $\eta$ and $\varepsilon$.\\

As we have seen in the previous section, the adjoint pair $(F,G)$ induces a Galois connection $(\varphi,\psi):Pr(\mathcal{A})\longrightarrow Pr(\mathcal{B})$. Similarly, the adjoint pair $(G^{op}, F^{op})$ induces a Galois connection $(\overline{\varphi},\overline{\psi}):Pr(\mathcal{B}^{op})\longrightarrow Pr(\mathcal{A}^{op})$. In this section we study the relation between these two Galois connections.\\

Recall that, for each abelian category $\mathcal{A}$ we define the duality assignment $\Delta_{\mathcal{A}}: Pr(\mathcal{A})\rightarrow Pr(\mathcal{A}^{op})$, such that $\Delta_{\mathcal{A}} :\tau \mapsto (\tau^*)^{op}$.
\vspace{.2in}

\begin{proposition}\label{Galoiscoprerad}
Let $(F,G):\mathcal{A}\longrightarrow \mathcal{B}$ be an adjoint pair between abelian categories, and let $(G^{op},F^{op}):\mathcal{B}^{op}\longrightarrow \mathcal{A}^{op}$ be the corresponding adjoint pair between the opposite categories. Consider the Galois conections induced by these adjoint pairs:
$$(\varphi,\psi):Pr(\mathcal{A})\longrightarrow Pr(\mathcal{B}),\quad\quad\quad  (\overline{\varphi},\overline{\psi}):Pr(\mathcal{B}^{op})\longrightarrow Pr(\mathcal{A}^{op})$$

Then the following diagrams commute:

\begin{equation*}\
\begin{aligned}
\xymatrix{Pr(\mathcal{A})\ar[r]^{\varphi}\ar[d]_{\Delta_{\mathcal{A}}} & Pr(\mathcal{B})\ar[d]^{\Delta_{\mathcal{B}}} & & Pr(\mathcal{A})\ar[d]_{\Delta_{\mathcal{A}}} & Pr(\mathcal{B})\ar[d]^{\Delta_{\mathcal{B}}}\ar[l]_{\psi}\\
Pr(\mathcal{A}^{op})\ar[r]_{\overline{\psi}} & Pr(\mathcal{B}^{op}) & & Pr(\mathcal{A}^{op}) & Pr(\mathcal{B}^{op})\ar[l]^{\overline{\varphi}}}
\end{aligned}
\end{equation*}
In other words,
\begin{enumerate}
    \item[(a)] For each $\tau\in Pr(\mathcal{A}),\  \Delta_{\mathcal{B}}(\varphi(\tau))=\overline{\psi}(\Delta_{\mathcal{A}}(\tau))$ 
    \item[(b)] For each $\sigma\in Pr(\mathcal{B}),\  \Delta_{\mathcal{A}}(\psi(\sigma))=\overline{\varphi}(\Delta_{\mathcal{B}}(\sigma))$ 
\end{enumerate}

\end{proposition}

\begin{proof}
$(a)$ Let $\tau\in Pr(\mathcal{A})$ and consider the corresponding SQ-sequence:

\centerline{$\xymatrix{T\ar@{^(->}[r]^{\tau} & 1_{\mathcal{A}}\ar@{->>}[r]^{\tau^*} & T^*}.$}

Then, by definition, $\varphi(\tau)=\mathrm{Im}\Big(\xymatrix{FTG\ar[r]^{F\tau G} & FG\ar[r]^{\varepsilon} & 1_{\mathcal{B}}}\Big)$, thus $$\varphi(\tau)^*=\mathrm{Coker}\Big(\xymatrix{FTG\ar[r]^{F\tau G} & FG\ar[r]^{\varepsilon} & 1_{\mathcal{B}}}\Big)$$ 
Therefore, in $\mathrm{Fun}(\mathcal{B}^{op},\mathcal{B}^{op})$, we have 
\begin{align*}
    \Delta_{\mathcal{B}}(\varphi(\tau))&=(\varphi(\tau)^*)^{op}\\ &=\mathrm{Ker}\Big(\xymatrix{1_{\mathcal{B}^{op}}\ar[r]^{\varepsilon'} & F^{op}G^{op}\ar[rr]^(.5){F^{op}\tau^{op}G^{op}} & &  F^{op}T^{op}G^{op}}\Big)\\
    &=\overline{\psi}((\tau^*)^{op})=\overline{\psi}(\Delta_{\mathcal{A}}(\tau))
\end{align*}
The third equality results considering that in $\mathcal{A}^{op}$ we have $\tau^{op}=((\tau^*)^{op})^*$, and by definition of $\overline{\psi}$ as part of the Galois connection induced by the adjoint pair $(G^{op},F^{op})$.\\

$(b)$ Let $\sigma\in Pr(\mathcal{B})$ and consider the corresponding SQ-sequence in $\mathcal{B}$:

\centerline{$\xymatrix{S\ar@{^(->}[r]^{\sigma} & 1_{\mathcal{B}}\ar@{->>}[r]^{\sigma^*} & S^*}$}

Then, by definition, $$\psi(\sigma)=\mathrm{Ker}\Big(\xymatrix{1_{\mathcal{A}}\ar[r]^{\eta} & GF\ar[r]^(.4){G\sigma^*F} &  GS^* F)}\Big)$$ thus in $\mathrm{Fun}(\mathcal{A}^{op},\mathcal{A}^{op})$, we have 
\begin{align*}
(\psi(\sigma))^{op}&=\mathrm{Coker}\Big(\xymatrix{G^{op}(S^*)^{op}F^{op}\ar[rrr]^{G^{op}(\sigma^*)^{op}F^{op}} & & & G^{op}F^{op}\ar[r]^{\eta'} & 1_{\mathcal{A}^{op}}}\Big)\\
&= \overline{\varphi}((\sigma^*)^{op})^*=\overline{\varphi}(\Delta_{\mathcal{B}}(\sigma))^*
\end{align*}

The last equality results by definition of $\overline{\varphi}$ as part of the Galois connection induced by the adjoint pair $(G^{op},F^{op})$, and considering that $(\sigma^*)^{op}$ is a preradical in $Pr(\mathcal{B}^{op})$.\\

Therefore, we have a SQ-sequence in $\mathcal{A}^{op}$:

\centerline{$\xymatrix{\ T\ \ar[r]^{\overline{\varphi}(\Delta_{\mathcal{B}}(\sigma))} & \ \ 1_{\mathcal{A}^{op}}\ \ar@{->>}[r]^{(\psi(\sigma))^{op}} & \ T^*\ }$}
\vspace{.2in}
which induces a corresponding SQ-sequence in $\mathcal{A}$:

\centerline{$\xymatrix{\ (T^*)^{op}\ \ar[r]^{\psi(\sigma)} & \ 1_{\mathcal{A}}\ \ \ar@{->>}[r]^{\overline{\varphi}(\Delta_{\mathcal{B}}(\sigma))^{op}} & \ T^{op}\ }$}
\vspace{.2in}
Therefore $\psi(\sigma)^*=\overline{\varphi}(\Delta_{\mathcal{B}}(\sigma))^{op}$, so that $\Delta_{\mathcal{A}}(\psi(\sigma))=\overline{\varphi}(\Delta_{\mathcal{B}}(\sigma)).$

\end{proof}

\section{Alpha and omega preradicals}
As in the case of preradicals for a ring $R$, we can define alpha and omega preradicals in abelian categories, which have a special importance for all preradicals. In particular, these two classes of preradicals are connected respectively under the morphisms that are part of the Galois connection induced by an adjoint pair.\\

From now on, we assume that the abelian category $\mathcal{A}$ is also complete and cocomplete, i.e., bicomplete. Let $N,M$ be objects in $\mathcal{A}$.  We define a functor $\mathbb{A}_{M}^{N}\colon\mathcal{A}\longrightarrow \mathcal{A}$
\noindent as follows:
\begin{enumerate}[(a)]
\item  For any object $L\in \mathcal{A}$ we define $\mathbb{A}_{M}^{N}(L):=N^{(\mathrm{Hom}_{\mathcal{A}}(M,L))}$, the coproduct of copies of $N$ indexed by the set $\mathrm{Hom}_{\mathcal{A}}(M,L)$.

\item For any morphism $\lambda:L\longrightarrow L'$ we set:
$$\mathbb{A}_{M}^{N}(\lambda):N^{(\mathrm{Hom}_{\mathcal{A}}(M,L))}\longrightarrow N^{(\mathrm{Hom}_{\mathcal{A}}(M,L'))}$$
where $\mathbb{A}_{M}^{N}(\lambda)$ is the unique morphism such that for every $f\in \mathrm{Hom}_{\mathcal{A}}(M,L)$ the following diagram commutes:
\begin{equation}\label{defANM}
\begin{aligned}
\xymatrix{N^{(\mathrm{Hom}_{\mathcal{A}}(M,L))}\ar[rr]^{\mathbb{A}_{M}^{N}(\lambda)} & & N^{(\mathrm{Hom}_{\mathcal{A}}(M,L'))}\\
& N\ar[ul]^{i_{f}}\ar[ur]_{i'_{\lambda\circ f}}}
\end{aligned}
\end{equation}
where $i_{f}$ is the canonical inclusion for $f\in \mathrm{Hom}_{\mathcal{A}}(M,L)$ and $i'_{\lambda\circ f}$ is the canonical inclusion for  $\lambda\circ f\in \mathrm{Hom}_{\mathcal{A}}(M,L')$.

\end{enumerate}
\vspace{.2in}

In other words, this functor is obtained as the composition $\mathbb{A}_{M}^{N}=N^{(-)}\circ \mathrm{Hom}_{\mathcal{A}}(M,-)$, where
$\mathrm{Hom}_{\mathcal{A}}(M,-):\mathcal{A}\longrightarrow \mathrm{Set}$ is the covariant Hom functor
and $N^{(-)}:\mathrm{Set}\longrightarrow \mathcal{A}$ is the covariant functor defined by the coproduct and the universal property of the coproduct in $\mathcal{A}$. We need a natural transformation to define the corresponding alpha preradical.\\

\begin{proposition}\label{defagot}
Let $h:N\longrightarrow M$ be any morphism in $\mathcal{A}$. For each $L\in \mathcal{A}$ we define the morphism $[\agot_{h}]_{L}$ as the unique morphism such that for all $f\in \mathrm{Hom}_{\mathcal{A}}(M,L)$ the following diagram commutes:

\begin{equation}\label{agotL}
\begin{aligned}
\xymatrix{N^{(\mathrm{Hom}_{\mathcal{A}}(M,L))}\ar[rr]^{[\agot_{h}]_{L}} & & L\\
& N\ar[ul]^{i_{f}}\ar[ur]_{f\circ h}}
\end{aligned}
\end{equation}

\noindent Then $\agot_{h}:\mathbb{A}_{M}^{N}\longrightarrow 1_{\mathcal{A}}$ is a natural transformation.
\end{proposition}
\begin{proof}
    
let us consider a morphism $\lambda:L\longrightarrow L'$. By definition of $\mathbb{A}_{M}^{N}$, we have diagram (\ref{defANM}). Also we have diagram (\ref{agotL}) and the corresponding diagram:

\begin{equation}\label{agotL'}
\begin{aligned}
\xymatrix{N^{(\mathrm{Hom}_{\mathcal{A}}(M,L'))}\ar[rr]^{[\agot_{h}]_{L'}} & & L'\\
& N\ar[ul]^{i_{\lambda\circ f}}\ar[ur]_{\lambda\circ f\circ h}}
\end{aligned}
\end{equation}

From these three diagrams we conclude that both morphisms $\lambda\circ [\agot_{h}]_{L}$ and $[\agot_{h}]_{L'}\circ\mathbb{A}_{M}^{N}(\lambda)$ complete for all $f\in\mathrm{Hom}(M,L)$ the commutative diagram:

\centerline{$\xymatrix{N^{(\mathrm{Hom}(M,L))}\ar@{-->}[rr]  & & L'\\
& N\ar[ul]^{i_{f}} \ar[ur]_{\lambda\circ f\circ h} & }$}

Therefore, by uniqueness in the universal property of the coproduct, we conclude that the following diagram commutes:

\centerline{$\xymatrix{
\mathbb{A}_{M}^{N}(L)\ar[d]_{\mathbb{A}_{M}^{N}(\lambda)}\ar[r]^{\ [\agot_{h}]_{L} } & L\ar[d]^{\lambda} \\
\mathbb{A}_{M}^{N}(L')\ar[r]_{\ [\agot_{h}]_{L'} } & L' } $}

\noindent which means that $\agot_{h}:\mathbb{A}_{M}^{N}\longrightarrow 1_{\mathcal{A}}$ is a natural transformation.

\end{proof}

\begin{definition}\label{definicionalpha}
Let $h:N\longrightarrow M$ be a morphism in $\mathcal{A}$. We define the {\bf alpha preradical} associated to $h$ as $\xymatrix{\alpha_{h}:=\mathrm{Im}(\agot_{h}):A_h\ar@{^(->}[r] & 1_{\mathcal{A}}}$ in $\mathcal{A}$.
 \end{definition}

 Thus we have for such morphism $h$ the following epi-mono factorization:

\begin{equation}\label{agotepimono}
\begin{aligned}
\xymatrix{\mathbb{A}_{M}^{N}\ar[rr]^{\agot_{h}}\ar@{->>}[dr]_{p_h} & & 1_{\mathcal{A}}\\
& A_h \ar@{^(->}[ur]_{\alpha_h}}
 \end{aligned}
 \end{equation}

 The following result will be useful for the next section.
\vspace{.2in}
 \begin{lemma}\label{alfaiso}
Let $h:N\longrightarrow M$ be an isomorphism in $\mathcal{A}$. Then $\alpha_h=\alpha_{1_M}$.
 \end{lemma}

 \begin{proof}
 For any isomorphism $h:N\longrightarrow M$, let $L\in\mathcal{A}$ and let $\delta:L\longrightarrow K$ be any morphism. We claim that $\delta\circ [\agot_{h}]_{L}=0$ if, and only if, $\delta\circ [\agot_{1_M}]_{L}=0$. So let us consider, for each morphism $f:M\rightarrow L$, the commutative triangle (\ref{agotL}) and the following corresponding commutative triangle for $1_M$:

\centerline{$
\xymatrix{M^{(\mathrm{Hom}_{\mathcal{A}}(M,L))}\ar[rr]^{[\agot_{1_M}]_{L}} & & L\\
& M\ar[ul]^{i_{f}}\ar[ur]_{f}}$}

Since $h$ is an isomorphism, we have that $\delta\circ f\circ h=0$ if, and only if, $\delta\circ f=0$. This implies that $\delta\circ [\agot_{h}]_{L}=0$ if, and only if, $\delta\circ [\agot_{1_M}]_{L}=0$, and therefore $\mathrm{Coker}([\agot_{h}]_{L})=\mathrm{Coker}([\agot_{1_M}]_L)$, and since this equality holds for each $L\in\mathcal{A}$, we conclude that $\alpha_h=\mathrm{Im}([\agot_{h}])=\mathrm{Im}([\agot_{1_M}])=\alpha_{1_M}$.

 \end{proof}
 
 We shall prove that if $(\varphi, \psi)$ is the Galois conection induced by an adjoint pair, alpha preradicals are mapped to alpha preradicals under $\varphi$. First we need some previous results.
\vspace{.2in}

\begin{lemma}\label{alphaXi}
Let $(F,G):\mathcal{A}\longrightarrow \mathcal{B}$ be an adjoint pair between bicomplete abelian categories.
Let $h:N\longrightarrow M$ be a morphism in $\mathcal{A}$. Then, we have that
$\varphi(\alpha_{h})=\mathrm{Im}\Big(\xymatrix{F\mathbb{A}_{M}^{N}G\ar[r]^{F\agot_{h}G} & FG\ar[r]^{\epsilon} & 1_{\mathcal{B}}}\Big)$.
\end{lemma}
\begin{proof}
By definition of $\varphi$ we have $\varphi(\alpha_{h})=\mathrm{Im}\Big(\xymatrix{FA_{h}G\ar[r]^{F\alpha_{h}G} & FG\ar[r]^{\varepsilon} & 1_{\mathcal{B}}}\Big)$.

Since $F$ is left adjoint to $G$, it preserves epimorphisms. Therefore, for each $B\in \mathcal{B}$, $F([p_{h}]_{G(B)})$ is an epimorphism, and we have: 

{\vspace{-15pt}
\begin{align*}
&\varphi(\alpha_{h})(B)=\\
&= \mathrm{Im}\Big(\xymatrix{(FA_{h}G)(B)\ar[rr]^{F([\alpha_{h}]_{G(B)})} & & FG(B)\ar[r]^{\varepsilon_{B}} & B} \Big)\\
 &=\mathrm{Im}\Big(\xymatrix{(F\mathbb{A}_{M}^{N}G)(B)\ar[rr]^{F([p_{h}]_{G(B)})}  & & 
 (FA_hG)(B)\ar[rr]^{F([\alpha_{h}]_{G(B)})} & &
 FG(B)\ar[r]^{\varepsilon_{B}} & B} \Big)\\
&= \mathrm{Im}\Big(\xymatrix{(F\mathbb{A}_{M}^{N}G)(B)\ar[rr]^{F([\agot_{h}]_{G(B)})} & &  FG(B)\ar[r]^{\varepsilon_{B}} & B} \Big).
\end{align*}}

Therefore, $\varphi(\alpha_{h})=\mathrm{Im}\Big(\xymatrix{F\mathbb{A}_{M}^{N}G\ \ar[r]^{F\agot_{h}G} & FG\ar[r]^{\epsilon} & 1_{\mathcal{B}}}\Big)$.
\end{proof} 
\vspace{.2in}

\begin{lemma}\label{unisocop}
Let $(F,G):\mathcal{A}\longrightarrow \mathcal{B}$ be an adjoint pair between bicomplete abelian categories.
Let $M,N$ be objects in $\mathcal{A}$. Then, there exists a natural isomorphism $\gamma: \mathbb{A}_{F(M)}^{F(N)}\longrightarrow F\mathbb{A}_{M}^{N}G$ such that the following diagram commutes in $\mathrm{Fun}(\mathcal{B},\mathcal{B})$:

\begin{equation}\label{gammacomm}
\begin{aligned}
\xymatrix{F\mathbb{A}_{M}^{N}G\ar[rr]^{F\circ \agot_{h}\circ G} & & FG\ar[r]^{\epsilon} & 1_{\mathcal{B}}\\
\mathbb{A}_{F(M)}^{F(N)}\ar[urrr]_{\agot_{F(h)}}\ar[u]^{\gamma} & }
\end{aligned}
\end{equation}
\end{lemma}
\begin{proof}
Let $B\in\mathcal{B}$ and let us consider the coproduct that consists of all morphisms 

\centerline{$i_{g}:N\rightarrow  N^{(\mathrm{Hom}(M,G(B)))} $ for $g\in \mathrm{Hom}(M,G(B)).$}
\vspace{.2in}
\noindent Since $F$ is left adjoint, it preserves coproducts and we can consider the coproduct that consists of all morphisms

\centerline{$F(i_{g}):F(N)\rightarrow F\Big( N^{(\mathrm{Hom}(M,G(B)))}\Big)$ for $g\in \mathrm{Hom}(M,G(B)).$} 
\vspace{.2in}
\noindent On the other hand, we have the coproduct that consists of all morphisms

\centerline{$\Lambda_{f}:F(N)\rightarrow (F(N))^{(\mathrm{Hom}(F(M),B))}$ for $f\in\mathrm{Hom}(F(M),B).$} 
\vspace{.2in}
\noindent Being $(F,G)$ an adjoint pair, we have a bijection

\centerline{$\mathrm{Hom}(F(M),B)\longrightarrow \mathrm{Hom}(M,G(B)),\quad f\mapsto G(f)\circ \eta_{M}.$}
\vspace{.2in}
\noindent Thus, there exists an isomorphism $\gamma_{B}$ such that, for each $f\in\mathrm{Hom}(F(M),B)$ the following diagram commutes:

\centerline{$\xymatrix{F(N)^{(\mathrm{Hom}(F(M),B))}\ar[rr]^{\gamma_{B}} & & F\Big( N^{(\mathrm{Hom}(M,G(B)))}\Big)\\
& F(N)\ar[ul]^{\Lambda_{f}}\ar[ur]_{F(i_{G(f)\circ \eta_{M}})} & }$}
\vspace{.2in}
\noindent It is easy to verify that this defines a natural isomorphism:

\centerline{$\gamma: \mathbb{A}_{F(M)}^{F(N)}\longrightarrow F\mathbb{A}_{M}^{N}G.$}
\vspace{.2in}
On the other hand, by definition of $\agot_{h}:\mathbb{A}_{M}^{N}\longrightarrow 1_{\mathcal{A}}$, we have for all $f\in\mathrm{Hom}(F(M),B)$ the following commutative diagram:

\centerline{$\xymatrix{F\Big( N^{(\mathrm{Hom}(M,G(B)))}\Big)\ar[rr]^{F([\agot_h]_{GB})} & & F(G(B))\\
& F(N)\ar[ul]^{F(i_{G(f)\circ \eta_{M}})} \ar[ur]_{F(G(f)\circ \eta_M\circ h)} & }$}
\vspace{.2in}
\noindent And the naturality of $\varepsilon$ and a triangle identity gives for all $f\in\mathrm{Hom}(F(M),B)$ the commutative diagram:

\centerline{$\xymatrix{F(G(B))\ar[rr]^{\varepsilon_B}  & & B\\
& F(N)\ar[ul]^{F(G(f)\circ \eta_M\circ h)} \ar[ur]_{f\circ F(h)} & }$}
\vspace{.2in}
Joining the three last commutative diagrams, we see that both morphisms $[\agot_{F(h)}]_B$ and $\varepsilon_B\circ F([\agot_h]_{GB})\circ \gamma_B$ complete for all $f\in\mathrm{Hom}(F(M),B)$ the commutative diagram:

\centerline{$\xymatrix{F(N)^{(\mathrm{Hom}(F(M),B))}\ar@{-->}[rr]  & & B\\
& F(N)\ar[ul]^{\Lambda_{f}} \ar[ur]_{f\circ F(h)} & }$}
\vspace{.2in}
Thus, by uniqueness in the universal property of the coproduct, we conclude that diagram (\ref{gammacomm}) commutes.

\end{proof} 

\begin{proposition}\label{varphialpha=alpha}
Let $(F,G):\mathcal{A}\longrightarrow \mathcal{B}$ be an adjoint pair between bicomplete abelian categories and let $(\varphi,\psi):Pr(\mathcal{A})\longrightarrow Pr(\mathcal{B})$ be the induced Galois connection. If $h:N\longrightarrow M$ is any morphism in $\mathcal{A}$, then $\varphi(\alpha_{h})=\alpha_{F(h)}$.
\end{proposition}
\begin{proof}
Since $\gamma: \mathbb{A}_{F(M)}^{F(N)}\longrightarrow F\mathbb{A}_{M}^{N}G$ is an isomorphism, by Lemma \ref{unisocop} we have that $\alpha_{F(h)}=\mathrm{Im}(\agot_{F(h)})= \mathrm{Im}\Big(\xymatrix{F\mathbb{A}_{M}^{N}G\ar[rr]^{F\circ \agot_{h}\circ G} & & FG\ar[r]^{\epsilon} & 1_{\mathcal{B}}}\Big)=\varphi(\alpha_{h})$, the last equality by Lemma \ref{alphaXi}.

\end{proof} 

Now our goal is to prove that alpha preradicals associated to isomorphisms are idempotent. We need some previous results. The first one is a sufficient condition for a preradical to be idempotent.
\vspace{.2in}

\begin{lemma}\label{idempotente1}
Let $\sigma: H\longrightarrow 1_{\mathcal{A}}$ be a natural transformation and $\xymatrix{\tau: T\ar@{^(->}[r] & 1_{\mathcal{A}}}$ be such that $\tau=\mathrm{Im}(\sigma)$.  If $H\tau$ is an epimorphism in $\mathrm{Fun}(\mathcal{A},\mathcal{A})$, then $\tau$ is idempotent.
\end{lemma}
\begin{proof}
Consider the following diagram 

\centerline{$\xymatrix{ & T^{2}\ar@{^(->}[dr]^{\tau T}\ar@{}[d]|{II} & \\
HT\ar[rr]^{\sigma T}\ar[dr]_{H\tau}\ar@{->>}[ur]^{\pi\tau} & & T\ar@{^(->}[dr]^{\tau}\ar@{}[d]|{I} & \\
& H\ar[rr]_{\sigma}\ar@{->>}[ur]^{\pi}\ar@{}[u]|{III} & & 1_{\mathcal{A}}}$}
\vspace{.2in}
\noindent constructed as follows. Triangle $I$ is just the factorization of $\sigma$ through its image. Triangle $II$ is obtained by composing with $T$ by the right to triangle $I$, so it is commutative. The square formed by triangles $III$ and $I$ is commutative by the naturality of $\sigma$. Thus $\tau \circ \pi \circ H\tau=\sigma \circ H\tau=\tau\circ \sigma T $, and since $\tau$ is a monomorphism, $\pi \circ H\tau=\sigma T$, so that triangle $III$ is also commutative. If $H\tau$ is an epimorphism, then $\sigma T$ is an epimorphism, so that $\tau T$ is an epimorphism and hence an isomorphism. This proves that $\tau$ is idempotent.
\end{proof} 
\vspace{.2in}
\begin{lemma}\label{isobonh}
Let $h:N\longrightarrow M$ be an isomorphism in $\mathcal{A}$. Then there exists a natural isomorphism

\centerline{$\Phi:\mathrm{Hom}_{\mathcal{A}}(M, A_{h}(-))\longrightarrow \mathrm{Hom}_{\mathcal{A}}(M,-)$.}

\end{lemma}
\begin{proof}
Certainly, by composing $\alpha_h:A_h\rightarrow 1_{\mathcal{A}}$ with the functor $\mathrm{Hom}_{\mathcal{A}}(M,-)$, we obtain a natural transformation:

\centerline{$\Phi:\mathrm{Hom}_{\mathcal{A}}(M, A_{h}(-))\longrightarrow \mathrm{Hom}_{\mathcal{A}}(M,-)$.}

\noindent such that for each $L\in\mathcal{A}$, $\Phi_L: g\mapsto [\alpha_{h}]_{L} \circ g$.\\

Recall that for this $L$ we have $[\alpha_h]_L=\text{Im}[\agot_h]_L$, and by definition of $\agot_h$ we have the commutative diagram:
\begin{equation}\label{agot}
\begin{aligned}
 \xymatrix{\mathbb{A}_{M}^{N}(L)\ar[rr]^{[\agot_{h}]_{L}} & & L\\
& N\ar[ul]^{i_{f}}\ar[ur]_{f\circ h}}      
\end{aligned}
\end{equation}

\noindent for all $f\in \mathrm{Hom}_{\mathcal{A}}(M,L)$. Since $[\agot_{h}]_{L}$ is a monomorphism, then $\Phi_L$ is a monomorphism. Now let $f\in \mathrm{Hom}_{\mathcal{A}}(M,L)$. By diagram (\ref{agot}), and since $h$ is an isomorphism, we have
$f=[\agot_{h}]_{L}\circ i_{f}\circ h^{-1}$. Then, by Lemma \ref{lemimker} we have that
$$\mathrm{Im}(f)= \mathrm{Im}([\agot_{h}]_{L}\circ i_{f}\circ h^{-1})\leq \mathrm{Im}([\agot_{h}]_{L})=[\alpha_{h}]_{L}$$

Therefore $\text{Im}(f)$, and thus $f$, factors through $[\alpha_{h}]_{L}$. We conclude that $\Phi_L$ is an isomorphism.
\end{proof} 

Next result will be useful to characterize idempotent preradicals.
\vspace{.2in}
\begin{proposition}\label{alphaidemhiso}
If $h:N\longrightarrow M$ is an isomorphism in $\mathcal{A}$ then $\alpha_{h}$ is idempotent.
\end{proposition}
\begin{proof}
Let us consider the natural isomorphism $\Phi$ defined in Lemma \ref{isobonh}, and the composition of $\phi$ by the left side with the functor $N^{(-)}:\mathrm{Set}\longrightarrow \mathcal{A}$, which is thus a natural isomorphism $ N^{(-)}\Phi$. Notice that
$$N^{(-)}\circ \mathrm{Hom}_{\mathcal{A}}(M, -)=\mathbb{A}_{M}^{N}$$ and that
$$N^{(-)}\circ \mathrm{Hom}_{\mathcal{A}}(M, A_{h}(-))=N^{(-)}\circ \mathrm{Hom}_{\mathcal{A}}(M, -)\circ A_{h}=\mathbb{A}_{M}^{N}\circ A_{h} $$
Therefore: $$N^{(-)} \Phi=\mathbb{A}_{M}^{N}\alpha_{h}: \mathbb{A}_{M}^{N}\circ A_{h}\longrightarrow \mathbb{A}_{M}^{N}$$

\noindent with $\alpha_h=\text{Im}(\agot_h)$, so that the hypothesis of Lemma \ref{idempotente1} are satisfied and we conclude that $\alpha_{h}$ is idempotent.
\end{proof} 

\vspace{.2in}
Now, we shall give dual definitions to establish omega preradicals and their properties, and we shall study the relations between alpha and omega preradicals. We start considering, for any objects $N,M$ in $\mathcal{A}$, the functor which is the composition of contravariant functors $\mathbb{W}_{M}^{N}:=M^{-}\circ \mathrm{Hom}_{\mathcal{A}}(-,N):\mathcal{A}\longrightarrow \mathcal{A}$ where
$\mathrm{Hom}_{\mathcal{A}}(-,N):\mathcal{A}\longrightarrow \mathrm{Set}$ is the contravariant Hom functor
and the contravariant functor $M^{-}:\mathrm{Set}\longrightarrow \mathcal{A}$ is defined by the product and the universal property of the product in $\mathcal{A}$.\\

Now, for any morphism $k:N\longrightarrow M$ in $\mathcal{A}$, we define a natural transformation
\[\wneg_{k}\colon 1_{\mathcal{A}}\longrightarrow  \mathbb{W}_{M}^{N}\]

\noindent as follows: for any $L\in \mathcal{A}$ we set
$[\wneg_{k}]_{L}:L\longrightarrow M^{ \mathrm{Hom}_{\mathcal{B}}(L,N)}$ as the unique morphism such that for all $f\in \mathrm{Hom}_{\mathcal{A}}(L,N)$ the following diagram commutes:

\begin{equation}\label{wtransf}   
\begin{aligned}
\xymatrix{L \ar[rr]^{[\wneg_{k}]_{L}}\ar[dr]_{k\circ f} & &  M^{ \mathrm{Hom}_{\mathcal{A}}(L,N)}\ar[dl]^{\pi_{f}}\\
& M & }
\end{aligned}
\end{equation}

\noindent It is easy to verify that this defines a natural transformation.

\vspace{.2in}
\begin{definition}\label{definitionomega}
Let $k:N\longrightarrow M$ be a morphism in $\mathcal{A}$. We define the \textbf{omega preradical} associated to $k$ as $\xymatrix{\omega_k:=\mathrm{Ker}(\wneg_{k}): \Omega_{k}\ar@{^(->}[r] & 1_{\mathcal{A}}}$ in $\mathcal{A}$.
\end{definition}
\vspace{.2in}
The following result establishes the duality between alpha and omega preradicals. Recall the assignment $\Delta: Pr(\mathcal{A})\rightarrow Pr(\mathcal{A}^{op})$, defined prior to Proposition \ref{delta}.
\vspace{.2in}

\begin{proposition}\label{omegaalfa}
Let $k:N\longrightarrow M$ be any morphism in $\mathcal{A}$. Then:
\begin{enumerate}
    \item [(a)] $\Delta(\alpha_k)=\omega_{k^{op}}$
    \item [(b)] $\Delta(\omega_k)=\alpha_{k^{op}}$
\end{enumerate}

\end{proposition}
\begin{proof}
To prove the equality in $(a)$, let $k:N\longrightarrow M$ be a morphism in $\mathcal{A}$ and consider the morphism $k^{op}:M\longrightarrow N$ in $\mathcal{A}^{op}$. Notice that the functor $(\mathbb{A}_{N}^{M})^{op}:\mathcal{A}^{op}\rightarrow \mathcal{A}^{op}$ maps any object $L$ to the product $N^{\mathrm{Hom}_{\mathcal{A}^{op}}(L,M)}$ in $\mathcal{A}^{op}$ and for each morphism $\lambda':L'\rightarrow L$ in $\mathcal{A}^{op}$, this functor maps $\lambda'$ to the morphism that completes the following diagram, which is obtained by reversing diagram (\ref{defANM}):

\centerline{\xymatrix{N^{\mathrm{Hom}_{\mathcal{A}}(L,M)}\ar[dr]^{p_{f'}} & & N^{\mathrm{Hom}_{\mathcal{A}}(L',M)}\ar[ll]_{(\mathbb{A}_{M}^{N})^{op}(\lambda')}\ar[dl]_{p'_{f'\circ \lambda'}}\\
& N}}

This is precisely the definition of the functor $\mathbb{W}_{N}^{M}:\mathcal{A}^{op}\rightarrow \mathcal{A}^{op}$. Thus, $(\mathbb{A}_{N}^{M})^{op}=\mathbb{W}_{N}^{M}$. Furthermore, considering diagram (\ref{wtransf}) the natural transformation $\wneg_{k^{op}}$ for each $L'\in\mathcal{A}^{op}$ is given by the following diagram in $\mathcal{A}^{op}$:

\centerline{\xymatrix{L' \ar[rr]^{[\wneg_{k^{op}}]_{L'}}\ar[dr]_{k^{op}\circ f'} & &  N^{ \mathrm{Hom}_{\mathcal{A}^{op}}(L',M)}\ar[dl]^{\pi_{f'}}\\
& N & }}

Thus the opposite of this natural transformation $\wneg_{k^{op}} : 1_{\mathcal{A}^{op}}\longrightarrow (\mathbb{A}_{N}^{M})^{op}$ in $\mathrm{Fun}(\mathcal{A},\mathcal{A})$ is $(\wneg_{k^{op}})^{op} : \mathbb{A}_{N}^{M}\longrightarrow 1_{\mathcal{A} }$, whose definition coincides with the natural transformation $\agot_{k}:\mathbb{A}_{N}^{M}\longrightarrow 1_{\mathcal{A}}$, i.e., $\agot_{k}=(\wneg_{k^{op}})^{op}$. Therefore:

$$\Delta(\alpha_k)=(\alpha_k^*)^{op}=(\text{Coker}(\agot_{k}))^{op}=\text{Ker}(\wneg_{k^{op}})=\omega_{k^{op}}$$

\vspace{.1in}
To prove equality in $(b)$ consider the inverse assignment $\Delta': Pr(\mathcal{A}^{op})\rightarrow Pr(\mathcal{A})$. By $(a)$ we have $\Delta'(\alpha_{k^{op}})=\omega_{k}$, from which $(b)$ follows.

\end{proof}

As an immediate consequence, we can prove the dual of Proposition \ref{alphaidemhiso}.
\vspace{.2in}

\begin{proposition}\label{omegaradical}
If $k:N\longrightarrow M$ is an isomorphism in $\mathcal{A}$ then $\omega_{k}$ is a radical.
\end{proposition}
\begin{proof}
If  $k:N\longrightarrow M$ is an isomorphism in $\mathcal{A}$ then  $k^{op}:M\longrightarrow N$ is an isomorphism in $\mathcal{A}^{op}$. By Proposition \ref{alphaidemhiso}, $\alpha_{k^{op}}$ is an idempotent in $ Pr(\mathcal{A}^{op})$. By Proposition \ref{omegaalfa}, $\omega_k=\Delta_{\mathcal{A}^{op}}(\alpha_{k^{op}})$, and by Proposition \ref{radidem}, it is a radical in $ Pr(\mathcal{A})$.

\end{proof}

\vspace{.1in}
Now we are ready to prove that omega preradicals are mapped to omega preradicals under the morphism $\psi$ of the Galois connection.

\begin{proposition}\label{psiomega}
Let $(F,G):\mathcal{A}\longrightarrow \mathcal{B}$ be an adjoint pair between bicomplete abelian categories and
let $(\varphi,\psi):Pr(\mathcal{A})\longrightarrow Pr(\mathcal{B})$ be the induced Galois connection. If $k:N\longrightarrow M$ is any morphism in $\mathcal{B}$ then $\psi(\omega_{k})=\omega_{G(k)}$.
\end{proposition}
\begin{proof}
For the adjoint pair $(F,G)$, let $(G^{op},F^{op}):\mathcal{B}^{op}\longrightarrow \mathcal{A}^{op}$ be the corresponding adjoint pair between the opposite categories, and let 
$(\overline{\varphi},\overline{\psi}):Pr(\mathcal{B}^{op})\longrightarrow Pr(\mathcal{A}^{op}) $ be the Galois connection induced by  $(G^{op},F^{op})$. Let $k:N\longrightarrow M$ be a morphism in $\mathcal{B}$. By Proposition \ref{Galoiscoprerad}, Proposition \ref{omegaalfa} and Proposition \ref{varphialpha=alpha}, respectively, we have:
$$\Delta_{\mathcal{A}}(\psi(\omega_k))=\overline{\varphi}(\Delta_{\mathcal{B}}(\omega_k))=\overline{\varphi}(\alpha_{k^{op}})=\alpha_{G^{op}(k^{op})} $$

Therefore, again by Proposition \ref{omegaalfa}:
$$\psi(\omega_k)=\Delta_{\mathcal B}(\alpha_{G^{op}(k^{op})})=\omega_{G(k)}$$
\end{proof}

\section{$Pr(\mathcal{A})$ as a complete lattice}
%PONER NOCIONE BASICAS DE REICULAS AQUI O EN LOS PRELIMINARES.\\

We know that if $\mathcal{A}$ is an abelian category then $Pr(\mathcal{A})$ is a partially ordered collection. Now let us assume, from now on, that in addition to $\mathcal{A}$ being bicomplete, we assume that it is locally small. Our first aim in this section is to prove that with these conditions $Pr(\mathcal{A})$ is a complete lattice, in the sense that each collection $\mathcal{C}$ (not necessarily a set) of preradicals in $Pr(\mathcal{A})$ has a least upper bound, or \textit{join}, denoted by $\bigvee \mathcal{C}$. And thus, each collection of preradicals has also a greatest lower bound, or \textit{meet}, denoted by $\bigwedge\mathcal{C}$ . We will describe each one, and then we will use them to prove some results involving, on one hand, alpha preradicals and idempotent preradicals, and on the other hand, omega preradicals and radicals.
\vspace{.2in}

 So let us consider any collection of preradicals  $\left\{\xymatrix{\sigma_{i}: S_{i}\ar@{^(->}[r] & 1_{\mathcal{A}}}\right\}_{i\in I}$, where the index $I$ is a class (not necessarily a set). Thus, for each $A\in \mathcal{A}$ we have the collection $\left\{\xymatrix{S_{i}(A)\ar@{^(->}[r]^(.6){(\sigma_{i})_{A}} & A}\right\}_{i\in I}$ of subobjects of $A$. Since $\mathcal{A}$ is locally small, this collection is a set, although the index $I$ is not necessarily a set, so we can describe the same collection with an index $I_A\subseteq I$ which is indeed a set. In fact, $I_A$ can be such that $\left\{\xymatrix{S_{i}(A)\ar@{^(->}[r]^(.6){(\sigma_{i})_{A}} & A}\right\}_{i\in I_A}$ is a complete irredundant set of representatives of $\left\{\xymatrix{S_{i}(A)\ar@{^(->}[r]^(.6){(\sigma_{i})_{A}} & A}\right\}_{i\in I}$. In other words, for each $j\in I$ there exists a unique $i\in I_A$ such that $\xymatrix{S_{j}(A)\ar@{^(->}[r]^(.6){(\sigma_{j})_{A}} & A}$ and $\xymatrix{S_{i}(A)\ar@{^(->}[r]^(.6){(\sigma_{i})_{A}} & A}$ represent the same subobject. Since $\mathcal{A}$ is cocomplete, for each $A\in \mathcal{A}$ the coproduct $\bigoplus_{i\in {I_A}}S_{i}(A)$ exists. For each $j\in I$ and each $A\in\mathcal{A}$, let $\iota_j^A:S_j(A)\longrightarrow \bigoplus_{i\in {I_A}}S_{i}(A)$ be the natural inclusion in the coproduct, possibly composed with an isomorphism. Thus we have an endofunctor $S$ on $\mathcal{A}$, which we may denote by $S=\bigoplus_{i\in {I}}S_{i}$, such that $A\mapsto \bigoplus_{i\in {I_A}}S_{i}(A)$ and for any morphism  $g:A\to A'$, there is associated $S(g)$ which, by the universal property of the coproduct, is the unique morphism such that for each $j\in I$ the following diagram commutes:
\begin{equation}\label{funtorS}   
\begin{aligned}
\xymatrix{
S_{j}(A)\ar[d]_{S_{j}(g)}\ar@{^(->}[r]^{\iota_{j}^{A}\ } & \bigoplus_{i\in I_A}S_{i}(A)\ar[d]^{S(g)} \\
S_{j}(A')\ar@{^(->}[r]_{\iota_{j}^{A'}\ } &\bigoplus_{i\in I_{A'}}S_{i}(A') } 
\end{aligned}
\end{equation}

\noindent It is clear that $S$ is a functor. Notice that diagram (\ref{funtorS}) also proves that for each $j\in I$,  $\xymatrix{\iota_j:S_j\ar@{^(->}[r] & S}$ is a natural transformation. Now we will define a natural transformation which is the first step to define the meet of a collection of preradicals.
\vspace{.2in}
\begin{proposition}\label{conspresup}
Let ${\mathcal C} = \left\{\xymatrix{\sigma_{i}: S_{i}\ar@{^(->}[r] & 1_{\mathcal{A}}}\right\}_{i\in I}$ be a collection of preradicals and let $S$ be the functor defined above. For each $A\in \mathcal{A}$, let $I_A$ be the set defined above, and let $\sigma^{\mathcal C}_A$ be the unique morphism  which, by the universal property of the coproduct, completes for each $j\in I$ the following commutative diagram:

\begin{equation}\label{sigmaC}
\begin{aligned}
\xymatrix{\bigoplus_{i\in I_A}S_{i}(A)\ar[rr]^{\sigma^{\mathcal C}_{A}} & & A\\
& S_{j}(A)\ar[lu]^{\iota_{j}^{A}}\ar[ur]_{(\sigma_{j})_{A}} &}
\end{aligned}
\end{equation}

\noindent Then $\sigma^{\mathcal C}: S\longrightarrow 1_{\mathcal{A}}$ is a natural transformation.
\end{proposition}

\begin{proof}
Let $g:A\longrightarrow A'$ be a morphism in $\mathcal{A}$. For each $j\in I$ we have the diagram:

\begin{equation}\label{diagcono}   
\begin{aligned}
\xymatrix{ & \bigoplus_{i\in I_A}S_{i}(A)\ar[d]^{\sigma^{\mathcal C}_{A}}\ar@/^3pc/[ddd]^{S(g)}\\
S_{j}(A)\ar@{^(->}[r]^{(\sigma_{j})_{A}}\ar[d]_{H_{j}(g)}\ar@{^(->}[ur]^{\iota_{j}^{A}} & A\ar[d]^{g} \\
S_{j}(A')\ \ar@{^(->}[r]_{(\sigma_{j})_{A'}}\ar@{^(->}[dr]_{\iota_{j}^{A'}} & A' \\
 &  \bigoplus_{i\in I_{A'}}S_{i}(A)\ar[u]_{\sigma^{\mathcal C}_{A'}}}
 \end{aligned}
 \end{equation}

The external square is just commutative diagram (\ref{funtorS}). Both triangles are just the definition of $\sigma^{\mathcal C}$ and the left square commutes by the naturality of $\sigma_j$. We conclude that both morphisms $g\circ\sigma^{\mathcal C}_A$ and $\sigma^{\mathcal C}_{A'}\circ S(g)$ complete for all $j\in I$ the commutative triangle:

\centerline{$\xymatrix{\bigoplus_{i\in I_A}S_{i}(A)\ar@{-->}[rr] & & A'\\
& S_{j}(A)\ar[lu]^{\iota_{j}^{A}}\ar[ur]_{(\sigma_{j})_{A'}\circ S_j(g)} &}$}

Therefore, by uniqueness of the universal property of the coproduct, the right square of diagram (\ref{diagcono}) commutes, which means that $\sigma^{\mathcal C}$ is a natural transformation.
\end{proof}

\begin{proposition}\label{existesupremos}
Let $\mathcal{C}=\left\{\xymatrix{\sigma_{i}: S_{i}\ar@{^(->}[r] & 1_{\mathcal{A}}}\right\}_{i\in I}$ be a collection of preradicals. Then the preradical $\mathrm{Im} (\sigma^{\mathcal C})$ is the join of $\mathcal{C}$. In other words, $\mathrm{Im} (\sigma^{\mathcal C})=\bigvee_{i\in I}\sigma_{i}$.
\end{proposition}
\begin{proof}
For such a collection $\mathcal C$ of preradicals, by diagram (\ref{sigmaC}), and considering that $\sigma^{\mathcal C}$ is a natural transformation, as well as $\sigma_j$ and $\iota_j$, for each $j\in I$, we have that $\sigma_j$ factorizes through $\sigma^{\mathcal C}$, and thus, it factorizes through $\mathrm{Im} (\sigma^{\mathcal C})$. Therefore, $\sigma_j\leq \mathrm{Im} (\sigma^{\mathcal C})$ for each $j\in I$.

Now, let $\xymatrix{\tau: K\ar@{^(->}[r] & 1_{\mathcal{A}}}$ be a preradical such that $\sigma_{j}\leq \tau$ for all $j\in I$. Then for each $j\in I$ there exists a natural transformation $\gamma_{j}:S_{j}\longrightarrow K$ such that $\tau\circ\gamma_{j}=\sigma_{i}$.
On the other hand, we can define in a similar way to Proposition \ref{conspresup}, a natural transformation $\gamma$ such that for each $A\in\mathcal{A}$, $\gamma_A$ is the unique morphism which completes for each $j\in I$ the commutative diagram:

\begin{equation}\label{gamma}
\begin{aligned}
\xymatrix{\bigoplus_{i\in I_A}S_{i}(A)\ar[rr]^{\gamma_{A}} & & K(A)\\
& S_{j}(A)\ar[lu]^{\iota_{j}^{A}}\ar[ur]_{(\gamma_{j})_{A}} &}
\end{aligned}
\end{equation}

Therefore both morphisms $\sigma^{\mathcal C}_A$ and $\tau_A\circ\gamma_A$ complete for all $j\in I$ the commutative triangle:

\centerline{$\xymatrix{\bigoplus_{i\in I_A}S_{i}(A)\ar@{-->}[rr] & & A\\
& S_{j}(A)\ar[lu]^{\iota_{j}^{A}}\ar[ur]_{(\sigma_{j})_{A}} &}$}

And by uniqueness in the universal property of the coproduct, we conclude that $\sigma^{\mathcal C}_A=\tau_{A}\circ \gamma_{A}$. Since this happens for every $A\in\mathcal{A}$ then $\sigma^{\mathcal C}=\tau\circ \gamma$, Therefore, by Lemma \ref{lemimker}, we have $\mathrm{Im} (\sigma^{\mathcal C})\leq \mathrm{Im}(\tau)=\tau$. We conclude that $\mathrm{Im} (\sigma^{\mathcal C})=\bigvee_{i\in I}\sigma_{i}$.

\end{proof} 

It is known (see for example \cite{Stentrom}, Chapter III, Proposition 1.2) that if $L$ is a partially ordered set and every subset has a least upper bound then $L$ is a complete lattice. This fact can be extended to a partially ordered collection, such as $Pr(\mathcal{A})$. Thus by Proposition \ref{existesupremos}, we have the following result.
\vspace{.2in}

\begin{corollary}\label{prcomplelatice}
Let $\mathcal{A}$ be an
abelian category which is bicomplete and locally small. Then $Pr(\mathcal{A})$ is a complete lattice.
\end{corollary}
\qed

Since the assignment $\Delta$ defined before Proposition \ref{delta} is an anti-isomorphism of ordered collections, we have the following result.
\vspace{.2in}
\begin{lemma}\label{deltasupinf}
Let $\mathcal{A}$ be an abelian category that is locally small and bicomplete. Let us consider the assignment $\Delta: Pr(\mathcal{A})\rightarrow Pr(\mathcal{A}^{op})$. For any collection (not necessarily a set) $\{\sigma_i\}_{i\in I}$ in  $Pr(\mathcal{A})$ we have:
\vspace{.1in}
\begin{enumerate}
    \item [(a)] $\Delta(\bigvee_{i\in I}\sigma_{i})=\bigwedge_{i\in I}\Delta(\sigma_i)$
    \item [(b)] $\Delta(\bigwedge_{i\in I}\sigma_{i})=\bigvee_{i\in I}\Delta(\sigma_i)$
\end{enumerate}
\end{lemma}
\qed

We can describe the construction of the meet of any collection of preradicals $\left\{\xymatrix{\sigma_{i}: S_{i}\ar@{^(->}[r] & 1_{\mathcal{A}}}\right\}_{i\in I}$. In this case it is convenient to consider the collection of complete SQ-sequences:

\centerline{$\left\{\xymatrix{ S_i\ar@{^(->}[r]^{\sigma_i} & 1_{\mathcal{A}}\ar@{->>}[r]^{\sigma_i^*} & S_i^*}\right\}_{i\in I}$}

Again, for each $A\in\mathcal{A}$, let $I_A$ be the set contained in $I$ such that $\left\{\xymatrix{S_{i}(A)\ar@{^(->}[r]^(.6){(\sigma_{i})_{A}} & A}\right\}_{i\in I_A}$ is a complete irredundant set of representatives of the collection of subobjects $\left\{\xymatrix{S_{i}(A)\ar@{^(->}[r]^(.6){(\sigma_{i})_{A}} & A}\right\}_{i\in I}$, and therefore $\left\{\xymatrix{A\ar@{->>}[r]^{(\sigma_i^*)_A\ } & S_i^*(A)}\right\}_{i\in I_A}$ is a complete irredundant set of representatives of the collection of quotient objects $\left\{\xymatrix{A\ar@{->>}[r]^{(\sigma_i^*)_A\ } & S_i^*(A)}\right\}_{i\in I}$. And we have an endofunctor $P$ on $\mathcal{A}$, such that $A\mapsto \prod_{i\in {I_A}}S_{i}^*(A)$ and for any morphism  $g:A\to A'$, there is associated $P(g)$ which, by the universal property of the product, is the unique morphism such that for each $j\in I$ the following diagram commutes:
\begin{equation}\label{funtorP}   
\begin{aligned}
\xymatrix{
\prod_{i\in I_A}S_{i}^*(A)\ar[d]_{P(g)}\ar[r]^{\ \pi_{j}^{A}} & S_{j}^*(A)\ar[d]^{S_j^*(g)} \\
\prod_{i\in I_{A'}}S_{i}^*(A')\ar[r]_{\ \pi_{j}^{A'}} &S_{j}^*(A') } 
\end{aligned}
\end{equation}

Again, we have for each $j\in I$, a natural transformation $\xymatrix{\pi_j:P\ar[r] & S_j^*}$. And also a natural transformation described by the dual of Proposition \ref{conspresup}, whose proof is similar.
\vspace{.2in}
\begin{proposition}\label{conspreinf}
Let ${\mathcal C} = \left\{\xymatrix{\sigma_{i}: S_{i}\ar@{^(->}[r] & 1_{\mathcal{A}}}\right\}_{i\in I}$ be a collection of preradicals and let $P$ be the functor defined above. For each $A\in \mathcal{A}$, let $I_A$ be the set defined as above, and let $\tau^{\mathcal C}_A$ be the unique morphism  which, by the universal property of the product, completes for each $j\in I$ the following commutative diagram:

\begin{equation}\label{tauC}
\begin{aligned}
\xymatrix{A\ar[rd]^{(\sigma_{j}^*)_{A}}\ar[rr]^{\tau^{\mathcal C}_{A}} & & \prod_{i\in I_A}S_{i}^*(A)\ar[dl]_{\pi_{j}^{A}}\\
& S_{j}^*(A) &}
\end{aligned}
\end{equation}

\noindent Then $\tau^{\mathcal C}: 1_{\mathcal{A}}\longrightarrow P$ is a natural transformation.
\end{proposition}
\qed

Before constructing the meet of a collection of preradicals, we need another property of the duality assignment $\Delta$.
\vspace{.2in}
\begin{lemma}\label{deltakerim}
Let $\mathcal{A}$ be an abelian category.
\begin{enumerate}
    \item [(a)] Let $\delta: D\longrightarrow 1_{\mathcal{A}}$ be any natural transformation. Then $\Delta(\mathrm{Im}(\delta))=\mathrm{Ker}(\delta^{op})$
    \item [(b)]  Let $\rho: 1_{\mathcal{A}}\longrightarrow R$ be any natural transformation. Then $\Delta(\mathrm{Ker}(\rho))=\mathrm{Im}(\rho^{op})$
\end{enumerate}
\end{lemma}
\begin{proof}
To prove $(a)$, notice that if $\delta: D\longrightarrow 1_{\mathcal{A}}$ is a natural transformation in $\mathcal{A}$ then $\delta^{op}: 1_{\mathcal{A}^{op}}\longrightarrow D^{op}$ in $\mathcal{A}^{op}$. Let us consider the preradical $\xymatrix{\mathrm{Im}(\delta): D'\ar@{^(->}[r] & 1_{\mathcal{A}}}$. Then the SQ-sequence in $\mathcal{A}$:
$$\xymatrix{ D'\ar@{^(->}[r]^{\mathrm{Im}(\delta)} & 1_{\mathcal{A}}\ar@{->>}[r] & (D')^*}$$
\noindent induces the SQ-sequence in $\mathcal{A}^{op}$:
$$\xymatrix{ (D')^{*op}\ar@{^(->}[rr]^{\Delta(\mathrm{Im}(\delta))} & & 1_{\mathcal{A}^{op}}\ar@{->>}[r]^{\mathrm{Im}(\delta)^{op}} & (D')^{op}}$$
Therefore, $\Delta(\mathrm{Im}(\delta))=\mathrm{Ker}(\mathrm{Im}(\delta)^{op})=\mathrm{Ker}(\delta^{op})$.\\
The proof of $(b)$ is similar.
\end{proof}

The proof of the following result can be written in a similar way to that of Proposition \ref{existesupremos}, but we present another proof, using the duality of preradicals.
\vspace{.2in}
\begin{proposition}\label{existeinfimos}
Let $\mathcal{C}=\left\{\xymatrix{\sigma_{i}: S_{i}\ar@{^(->}[r] & 1_{\mathcal{A}}}\right\}_{i\in I}$ be a collection of preradicals. Then the preradical $\mathrm{Ker} (\tau^{\mathcal C})$ is the meet of $\mathcal{C}$. In other words, $\mathrm{Ker} (\tau^{\mathcal C})=\bigwedge_{i\in I}\sigma_{i}$.
\end{proposition}
\begin{proof}
For such a collection $\mathcal{C}$ of preradicals, let us consider the collection $\mathcal{C}'=\left\{\xymatrix{\Delta(\sigma_{i}): (S_{i}^*)^{op}\ar@{^(->}[r] & 1_{\mathcal{A}^{op}}}\right\}_{i\in I}$ of preradicals in $\mathcal{A}^{op}$. By Proposition \ref{conspresup}, we have the natural transformation $(\sigma')^{\mathcal C'}: S'\longrightarrow 1_{\mathcal{A}^{op}}$. And by Proposition \ref{existesupremos} and Lemma \ref{deltasupinf} we have in $Pr(\mathcal{A}^{op})$: 
\begin{equation}\label{imsupinf}
\begin{aligned}
\mathrm{Im} ((\sigma')^{\mathcal C'})=\bigvee_{i\in I}\Delta(\sigma_{i})=\Delta\left (\bigwedge_{i\in I}\sigma_{i}\right )
\end{aligned}  
\end{equation}

Now, by definition of the natural transformations $(\sigma')^{\mathcal C'}$ and $\tau^{\mathcal C}$, it is clear that $((\sigma')^{\mathcal C'})^{op}=\tau^{\mathcal C}$. Therefore, applying the inverse assignment $\Delta': Pr(\mathcal{A}')\longrightarrow Pr(\mathcal{A})$ to (\ref{imsupinf}) and by Lemma \ref{deltakerim} we have:

$$\mathrm{Ker} (\tau^{\mathcal C})=\Delta'(\mathrm{Im} ((\sigma')^{\mathcal C'}))=\bigwedge_{i\in I}\sigma_{i}$$

\end{proof}

The following result states that the collection of all idempotent preradicals is closed under arbitrary joins.
\vspace{.2in}

\begin{proposition}\label{supreidemp}
Let $\mathcal{C}=\left\{\xymatrix{\sigma_{i}:S_{i}\ar@{^(->}[r] & 1_{\mathcal{A}}}\right\}_{i\in I}$ be a collection of idempotent preradicals. Then  $\bigvee_{i\in I}\sigma_{i}$ is an idempotent preradical.
\end{proposition}
\begin{proof}
For such a collection $\mathcal{C}$, let $\xymatrix{\mu^{\mathcal{C}}:=\bigvee_{i\in I}\sigma_{i}:H\ar@{^(->}[r] &  1_{\mathcal{A}}}$. By Proposition \ref{existesupremos},  $\mu^{\mathcal{C}}=\mathrm{Im}(\sigma^{\mathcal{C}})$, where $\sigma^{\mathcal{C}}:S\longrightarrow 1_{\mathcal{A}}$ is the natural transformation described in Proposition \ref{conspresup}, and $S=\bigoplus_{i\in {I}}S_{i}$ is the funtor described in (\ref{funtorS}). By Lemma \ref{idempotente1}, to show that $\mu^{\mathcal{C}}$ is idempotent, it is enough to prove that $S\mu^{\mathcal{C}}$ is an epimorphism.\\

Let  $A\in \mathcal{A}$ and $j\in I$. Combining (\ref{sigmaC}) with the epi-mono factorization of $\sigma^{\mathcal{C}}_A$ we have the following commutative diagram: \vspace{5pt}

\centerline{\xymatrix{ & & S(A)\ar[d]^{\sigma^{\mathcal{C}}_A}\ar@{->>}[drr]^{\nu_{A}^{\mathcal{C}}} & &\\
S_{i}(A)\ar@{^(->}[rr]^{(\sigma_{i})_{A}}\ar@{^(->}[urr]^{\iota_{i}^{A}} & & A& & H(A)\ar@{_(->}[ll]_{\mu_{A}^{\mathcal{C}}} }}

\vspace{5pt}
Now, applying the funtor $S_{j}$ to this diagram we obtain the commutative diagram:\vspace{5pt}

\centerline{\xymatrix{ & & S_jS(A)\ar[d]^{S_i(\sigma^{\mathcal{C}}_A)}\ar[drr]^{S_i(\nu_{A}^{\mathcal{C}})} & &\\
S_jS_{i}(A)\ar[rr]^{S_i((\sigma_{i})_{A})}\ar[urr]^{S_i(\iota_{i}^{A})} & & S_i(A) & & S_iH(A)\ar[ll]_{S_i(\mu_{A}^{\mathcal{C}})} }}

\vspace{5pt}
Since $\sigma_{i}$ is idempotent, by Definition \ref{defiidempoten} $S_{i}\big((\sigma_{i})_{A}\big)$ is an isomorphism. Therefore  $S_{i}(\mu_{A}^{\mathcal{C}})$ is an epimorphism, for each $i\in I$. This implies that $S(\mu_{A}^{\mathcal{C}})=\bigoplus_{i\in I}S_{i}(\mu_{A}^{\mathcal{C}})$ is an epimorphism in $\mathcal{A}$. Therefore $S\mu^{\mathcal{C}}$ is an epimorphism in $\mathrm{Fun}(\mathcal{A},\mathcal{A})$. By Lemma \ref{idempotente1}, we conclude that $\mu^{\mathcal{C}}=\bigvee_{i\in I}\sigma_{i}$ is idempotent.
\end{proof} 

\begin{corollary}\label{idemcomplete}
Let $\mathcal{A}$ be an
abelian category which is bicomplete and locally small. Then $Idem(\mathcal{A})$ is a complete lattice.
\end{corollary}\qed
\vspace{.1in} 

Alpha preradicals ``generate'' all preradicals, as is stated in the following result.\\

\begin{proposition}\label{supalfa}
Let $\xymatrix{\tau: T\ar@{^(->}[r] & 1_{\mathcal{A}}}$ be any preradical in $\mathcal{A}$. Then $\tau=\bigvee_{M\in \mathcal{A}}\alpha_{\tau_{M}}$, where, for each $M\in \mathcal{A}$, $\tau_{M}:T(M)\longrightarrow M$ is the corresponding morphism in $\mathcal{A}$. 
\end{proposition}
\begin{proof}
Let us consider $\xymatrix{\tau: T\ar@{^(->}[r] & 1_{\mathcal{A}}}$. By Definition \ref{definicionalpha}, for each $M\in\mathcal{A}$ we have $\xymatrix{\alpha_{\tau_M}:=\mathrm{Im}(\agot_{\tau_M}):A_{\tau_M}\ar@{^(->}[r] & 1_{\mathcal{A}}}$, where $\agot_{\tau_M}:\mathbb{A}_{M}^{T(M)}\longrightarrow 1_{\mathcal{A}}$ is the natural transformation defined in Proposition \ref{defagot}. Let $C\in \mathcal{A}$. Recall that $\mathbb{A}_{M}^{T(M)}(C)=T(M)^{(\mathrm{Hom}_{\mathcal{A}}(M,C))}$. By definition of  $\agot_{\tau_M}$ and its epi-mono factorization, we have, respectively, commutative triangles $(I)$ (for each morphism $x:M\rightarrow C$) and $(II)$ of the following diagram. On the other hand, let us consider the funtor $S:=\bigoplus_{N\in \mathcal{A}}A_{\tau_{N}}$, and the natural transformation $\sigma^{\mathcal A}: S\longrightarrow 1_{\mathcal{A}}$ described in Proposition \ref{conspresup}. Let $\xymatrix{\mu^{\mathcal{A}}:=\bigvee_{M\in \mathcal{A}}\alpha_{\tau_{M}}:H\ar@{^(->}[r] &  1_{\mathcal{A}}}$. For this $C\in\mathcal{A}$ consider commutative triangles $(III)$ and $(IV)$ of the following diagram, which correspond, respectively, to the definition of $\sigma_C^{\mathcal A}$ and its epi-mono factorization. Also consider the morphism $\tau_C$ and the commutative square $(V)$ that comes from the naturality of $\tau$, for each morphism $x:M\rightarrow C$,

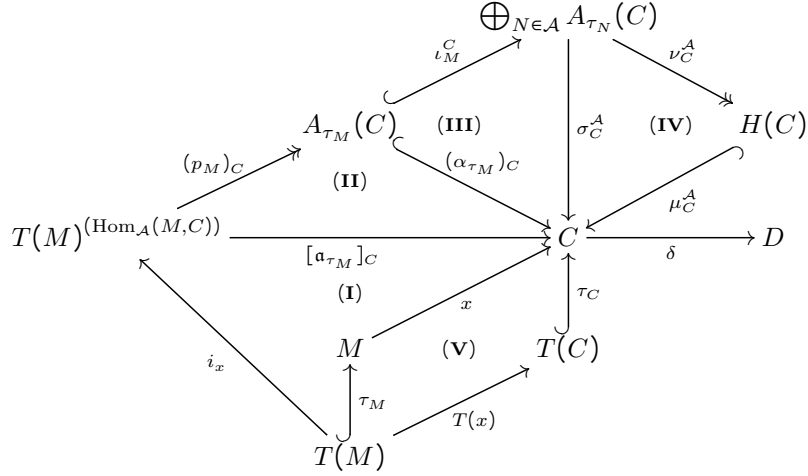
\centerline{$\xymatrix{& & \bigoplus_{N\in \mathcal{A}}A_{\tau_{N}}(C)\ar@{->>}[dr]^{\nu_{C}^{\mathcal A}}\ar[dd]^{\sigma_{C}^{\mathcal A}} & &\\
& A_{\tau_{M}}(C)\ar@{^(->}[dr]^{(\alpha_{\tau_{M}})_{C}}\ar@{^(->}[ur]^{\iota_{M}^{C}}\ar@{}[d]|{(\textbf{II})}\ar@{}[r]|{   (\textbf{III})} & &  H(C)\ar@{^(->}[dl]^{\mu_{C}^{\mathcal A}}\ar@{}[l]|{(\textbf{IV})} & \\ 
T(M)^{(\mathrm{Hom}_{\mathcal{A}}(M,C))}\ar[rr]_{[\agot_{\tau_{M}}]_{C}}\ar@{->>}[ur]^{(p_{M})_C} & & C\ar[r]_{\delta}   & D\\
  & M\ar[ur]_{x}\ar@{}[u]|{(\textbf{I})}\ar@{}[r]|{(\textbf{V})} &  T(C)\ar@{^(->}[u]_{\tau_{C}}\\
  & T(M) \ar@{^(->}[u]_{\tau_{M}}\ar[ur]_{T(x)}\ar@{}[uu]|{}\ar[luu]^{i_{x}} & }$}
\vspace{.2in}

Let $\delta:C\longrightarrow D$ be any morphism in $\mathcal{A}$. 
We claim that $\delta\circ \mu_{C}^{\mathcal A}=0$ if and only if $\delta\circ \tau_{C}=0$. Suppose first that $\delta\circ \mu_{C}^{\mathcal A}=0$. Following this commutative diagram, we have that $\delta\circ x\circ \tau_{M}=0$ for all $M\in\mathcal{A}$ and for all $x\in \mathrm{Hom}_{\mathcal{A}}(M,C)$. In particular, taking $M=C$ and $x=1_{C}$, we have $\delta\circ \tau_{C}=0$. Conversely, suppose that $\delta\circ \tau_{C}=0$. Then $\delta\circ x\circ \tau_{M}=0$  for all $M\in\mathcal{A}$ and for all $x\in \mathrm{Hom}_{\mathcal{A}}(M,C)$. Hence $\delta\circ [\agot_{\tau_{M}}]_{C}\circ i_{x}=0$ for all $x\in \mathrm{Hom}_{\mathcal{A}}(M,C)$. By the universal property of the coproduct we have $\delta\circ [\agot_{\tau_{M}}]_{C}=0$. Since $(p_{C})_M$ is an epimorphism we have that $\delta\circ (\alpha_{\tau_{M}})_{C}=0$. By the commutativity of $(III)$
 we have that $\delta\circ \sigma_{C}^{\mathcal A}\circ \iota_{M}^{C}=0$. This equality holds for each $M\in \mathcal{A}$, so by the universal property of the coproduct we have that $\delta\circ \sigma_{C}^{\mathcal A}=0$. By the commutativity of $(IV)$ and since $\nu_{C}^{\mathcal A}$ is an epimorphism, we have $\delta\circ \mu_{C}^{\mathcal A}=0$.\\
 
We conclude that $\tau_C=\mu_C^{\mathcal A}$ for all $C\in\mathcal{A}$. Therefore $\tau\simeq\mu^{\mathcal A}=\bigvee_{M\in \mathcal{A}}\alpha_{\tau_{M}}.$
 \end{proof}

Using Proposition \ref{supalfa}, in particular we can describe idempotent preradicals. For any preradical $\xymatrix{\tau: T\ar@{^(->}[r] & 1_{\mathcal{A}}}$ we define the following class of objects:\\

\centerline{$\mathbb{T}_{\tau}:=\{M\in \mathcal{A}\mid \tau_{M}:T(M)\longrightarrow M\,\,\text{is an isomorphism}\}$}

\vspace{.2in}
\begin{proposition}\label{idemsupalfa}
Let $\xymatrix{\tau: T\ar@{^(->}[r] & 1_{\mathcal{A}}}$
be any preradical of $\mathcal{A}$. Then $\tau$ is idempotent if, and only if, $\tau=\bigvee_{M\in \mathbb{T}_{\tau}}\alpha_{1_{M}}$.
\end{proposition}
\begin{proof}
Let $\xymatrix{\tau: T\ar@{^(->}[r] & 1_{\mathcal{A}}}$ be a preradical and suppose that it is idempotent. By Proposition \ref{supalfa} we have $\tau=\bigvee_{M\in \mathcal{A}}\alpha_{\tau_{M}}$. By Lemma \ref{alfaiso}, for each $M\in\mathbb{T}_{\tau}$ we have $\alpha_{\tau_{M}}=\alpha_{1_{M}}$. Therefore $\bigvee_{M\in \mathbb{T}_{\tau}}\alpha_{1_{M}}\leq\tau$. On the other hand, let $N\in\mathcal{A}$ and for each $L\in\mathcal{A}$ consider, for any morphisms $f:N\rightarrow L$ and $g:T(N)\rightarrow L$, the following diagrams:

\centerline{$
\xymatrix{T(N)^{(\mathrm{Hom}_{\mathcal{A}}(N,L))}\ar[rr]^{[\agot_{\tau_N}]_{L}} & & L\\
& T(N)\ar[ul]^{i_{f}}\ar[ur]_{f\circ\tau_N}}$}

\centerline{$
\xymatrix{T(N)^{(\mathrm{Hom}_{\mathcal{A}}(T(N),L))}\ar[rr]^{[\agot_{1_{T(N)}}]_{L}} & & L\\
& T(N)\ar[ul]^{i_{g}}\ar[ur]_{g}}$}

Now let $\delta: L\rightarrow K$ be any morphism, and suppose that $\delta\circ[\agot_{1_{T(N)}}]_{L}=0$. Then $\delta\circ g=0$ for each $g:T(N)\rightarrow L$. In particular $\delta\circ f\circ\tau_N=0$ for each $f:N\rightarrow L$. Therefore $\delta\circ [\agot_{\tau_N}]_{L}=0$. We conclude that $\mathrm{Coker}([\agot_{\tau_N}]_{L})\leq \mathrm{Coker}([\agot_{1_{T(N)}}]_{L})$, and thus $\alpha_{\tau_N}\leq\alpha_{1_{T(N)}}$. Notice that, since $\tau$ is idempotent, $T(N)\in\mathbb{T}_{\tau}$. This holds for each $N\in\mathcal{A}$, so we conclude that $\tau=\bigvee_{M\in \mathcal{A}}\alpha_{\tau_{M}}\leq\bigvee_{M\in \mathbb{T}_{\tau}}\alpha_{1_{M}}$. Thus, we have the equality.\\

Conversely, if $\tau=\bigvee_{M\in \mathbb{T}_{\tau}}\alpha_{1_{M}}$ then, by Proposition \ref{alphaidemhiso} we have that $\alpha_{1_M}$ is idempotent for each $M\in\mathbb{T}_{\tau}$. So by Proposition \ref{supreidemp} we conclude that $\tau$ is idempotent. 
\end{proof}

The dual results for radicals can be stated using the duality assignment $\Delta$.
\vspace{.2in}

\begin{proposition}\label{suprerad}
Let $\mathcal{C}=\left\{\xymatrix{\sigma_{i}:S_{i}\ar@{^(->}[r] & 1_{\mathcal{A}}}\right\}_{i\in I}$ be a collection of radicals. Then  $\bigwedge_{i\in I}\sigma_{i}$ is a radical.
\end{proposition}

\begin{proof}
For such a collection of radicals in $Pr(\mathcal{A})$, by Corollary \ref{radidem} we have, for each $i\in I$, that $\Delta(\sigma_i)$ is an idempotent preradical in $Pr(\mathcal{A}^{op})$. By Lemma \ref{deltasupinf} and Proposition \ref{supreidemp},  $\Delta(\bigwedge_{i\in I}\sigma_{i})=\bigvee_{i\in I}\Delta(\sigma_{i})$ is an idempotent preradical in $Pr(\mathcal{A}^{op})$. Therefore, again by Corollary \ref{radidem}, we conclude that $\bigwedge_{i\in I}\sigma_{i}$ is a radical. 
\end{proof}

\begin{corollary}\label{complete}
Let $\mathcal{A}$ be an
abelian category which is bicomplete and locally small. Then $Rad(\mathcal{A})$ is a complete lattice.
\end{corollary}
\qed

Dually to Proposition \ref{supalfa}, every preradical is ``cogenerated" by omega preradicals.\\

\begin{proposition}
Let $\xymatrix{\tau:T\ar@{^(->}[r] & 1_{\mathcal{A}}}$ be any preradical of $\mathcal{A}$. For each $M\in\mathcal{A}$, let $\tau_{M}:T(M)\longrightarrow M$ be the corresponding morphism in $\mathcal{A}$. Then 
$\tau=\bigwedge_{M\in \mathcal{A}}\omega_{\tau^*_{M}}$.

\end{proposition}
\begin{proof}
For such a preradical $\tau$ in $Pr(\mathcal{A})$, considering the duality assignment $\Delta$ related to $\mathcal{A}$, and by Proposition \ref{supalfa}, we have $\Delta(\tau)=\bigvee_{M\in \mathcal{A}^{op}}\alpha_{\Delta(\tau)_{M}}$, in $Pr(\mathcal{A}^{op})$. Therefore, considering the inverse duality assignment $\Delta'$ we have, by Lemma \ref{deltasupinf},
$$\tau=\Delta'(\bigvee_{M\in \mathcal{A}^{op}}\alpha_{\Delta(\tau)_{M}})=\bigwedge_{M\in \mathcal{A}}\Delta'(\alpha_{\Delta(\tau)_{M}})$$
Notice that, for each $M\in\mathcal{A}$, $(\Delta(\tau)_M)^{op}=((\tau^*)^{op}_M)^{op}=\tau^*_M$. On the other hand, by Proposition \ref{omegaalfa} we have $\Delta'(\alpha_{\Delta(\tau)_{M}})=\omega_{(\Delta(\tau)_{M})^{op}}=\omega_{\tau^*_M}$. We conclude that $\tau=\bigwedge_{M\in \mathcal{A}}\omega_{\tau^*_{M}}$. 
\end{proof}

Dually to Proposition \ref{idemsupalfa}, we consider the class of  objects:\\

\centerline{$\mathbb{F}_{\tau}:=\{M\in \mathcal{A}\mid \tau_M=0\}$}
\vspace{5pt}

\begin{proposition}\label{imfimarad}
Let $\xymatrix{\tau: T\ar@{^(->}[r] & 1_{\mathcal{A}}}$
be any preradical in $\mathcal{A}$. Then $\tau$ is a radical if, and only if, $\tau=\bigwedge_{M\in \mathbb{F}_{\tau}}\omega_{1_{M}}$.
\end{proposition}

\begin{proof}
Let $\tau$ be such a preradical, and again consider the duality assignment $\Delta$. By Corollary \ref{radidem} and Proposition \ref{idemsupalfa} we have that $\tau$ is a radical if, and only if,  $\Delta(\tau)$ is idempotent if, and only if, $\Delta(\tau)=\bigvee_{M\in \mathbb{T}_{\Delta(\tau)}}\alpha_{1_{M}}$ in $Pr(\mathcal{A}^{op})$. Notice that for each $M\in\mathcal{A}^{op}$, we have that $M\in\mathbb{T}_{\Delta(\tau)}$ if, and only if, $\Delta(\tau)_M$ is an isomorphism in $\mathcal{A}^{op}$ if, and only if, $\tau^*_M$ is an isomorphism in $\mathcal{A}$ if, and only if, $\tau_M=0$ if, and only if $M\in\mathbb{F}_{\tau}$ as an object in $\mathcal{A}$. Therefore, considering the inverse duality assignment $\Delta'$ we have, by Lemma \ref{deltasupinf}, that $\tau$ is a radical if, and only if:
$$\tau=\Delta'(\bigvee_{M\in \mathbb{T}_{\Delta(\tau)}}\alpha_{1_{M}})=\bigwedge_{M\in \mathbb{F}_{\tau}}\Delta'(\alpha_{1_{M}})=\bigwedge_{M\in \mathbb{F}_{\tau}}\omega_{1_{M}}$$
The last equality holds by Proposition \ref{omegaalfa}. 
\end{proof}

Now we are ready to study the way that any Galois connection induced by an adjoint pair between abelian categories preserves idempotent preradicals and radicals.
\vspace{.2in}

\begin{theorem}\label{galradidem}
Let $(F,G):\mathcal{A}\longrightarrow \mathcal{B}$ be an adjoint pair between locally small and bicomplete abelian categories and let $(\varphi,\psi):Pr(\mathcal{A})\longrightarrow Pr(\mathcal{B})$ be the corresponding induced Galois connection. Let $\tau\in Pr(\mathcal{A})$ and $\sigma\in Pr(\mathcal{B})$.
\begin{enumerate}
    \item If $\tau$ is idempotent then $\varphi(\tau)$ is a idempotent.
    \item If $\sigma$ is a radical then $\psi(\sigma)$ is a radical.
\end{enumerate}
\end{theorem}

\begin{proof}
\begin{enumerate}
\item Let $\tau\in Pr(\mathcal{A})$ and suppose thet $\tau$ is idempotent. By Proposition \ref{idemsupalfa} we have $\tau=\bigvee_{M\in \mathbb{T}_{\tau}}\alpha_{1_M}$. Since  $(\varphi,\psi):Pr(\mathcal{A})\longrightarrow Pr(\mathcal{B})$ is a Galois connection, $\varphi$ preserves arbitrary joins, and thus:

{\vspace{-10pt}
\begin{align*} 
\varphi(\tau)=\varphi\Big(\bigvee_{M\in \mathbb{T}_{\tau}}\alpha_{1_M}\Big)=\bigvee_{M\in \mathbb{T}_{\tau}}\varphi(\alpha_{1_M})=\bigvee_{M\in \mathbb{T}_{\tau}}\alpha_{1_{F(M)}}
\end{align*} \vspace{-8pt}}

\noindent where the last equality holds by  Proposition \ref{varphialpha=alpha}. By Proposition \ref{alphaidemhiso} and Proposition \ref{supreidemp}, we conclude that $\varphi(\tau)$ is an idempotent preradical. \\

\item Dually, let $\sigma\in Pr(\mathcal{B})$ and suppose thet $\sigma$ is a radical. By Proposition \ref{imfimarad} we have $\sigma=\bigwedge_{M\in \mathbb{F}_{\sigma}}\omega_{1_M}$. Being  $(\varphi,\psi)$ a Galois connection, $\psi$ preserves arbitrary meets, and thus:

{\vspace{-10pt}
\begin{align*} 
\psi(\sigma)=\psi\Big(\bigwedge_{M\in \mathbb{F}_{\sigma}}\omega_{1_M}\Big)=\bigwedge_{M\in \mathbb{F}_{\sigma}}\psi(\omega_{1_M})=\bigwedge_{M\in \mathbb{F}_{\sigma}}\omega_{1_{G(M)}}
\end{align*} \vspace{-8pt}}

\noindent where the last equality holds by  Proposition \ref{psiomega}. By Proposition \ref{omegaradical} and Proposition \ref{suprerad}, we conclude that $\psi(\sigma)$ is a radical. 
\end{enumerate}
 
\end{proof}

\section{Examples}

\begin{example}
Let $K$ be any field and let $\Lambda=K\mathbb{A}_2$ be the path algebra of the quiver $\mathbb{A}_2:\xymatrix{e_1\ar[r] & e_2}$. Then
$\Lambda\cong\begin{pmatrix}
 K & 0\\
 K & K
\end{pmatrix}$. 

There are three non-isomorphic indecomposable modules, which are elements of the set $\Lambda\mathrm{-ind}=\{ S_1,P,S_2\}$, where $S_1=\Lambda e_1$ is the projective simple module, $P=\Lambda e_2$ and $S_2=P/rP$ is the injective simple module, where $rP$ is the Jacobson radical of $P$. The Auslander-Reiten diagram of $\Lambda$ is:
$$
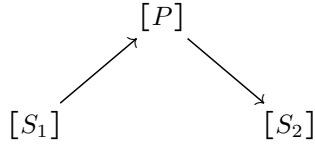
\xymatrix{& [P]\ar[dr] & \\
[S_{1}]\ar[ur] & & [S_{2}]} $$

Consider $\Lambda\mathrm{-Mod}$, the category of left $\Lambda$-modules, and $\Lambda\mathrm{-mod}$, the full subcategory of finitely generated left $\Lambda$-modules, which of course contains $\Lambda\mathrm{-ind}$. Any preradical in $\Lambda$ is determined uniquely by its value on each of the elements of $\Lambda\mathrm{-ind}$. In fact, each preradical can be denoted as $\sigma=[N_2,N,N_1]$, if $\sigma(S_2)=N_2$, $\sigma(P)=N$ and $\sigma(S_1)=N_1$. In \cite[Example 4.13]{FernandezMartelo} all eight preradicals in $\Lambda$ are described as follows:
\begin{align*}
    0_A &=[0,0,0],\  1_A=[S_1,P,S_2],\\
    \alpha^{S_1}_{S_1} &=\omega^{S_2}_0=[0,S_1,S_1],\\
    \iota_0 &=[S_2,S_1,S_1],\  \rho_1=[0,S_1,0]\\
    \gamma_0&=[S_2,0,0], \gamma_1=[S_2,P,0],\ \xi=[S_2,S_1,0]
\end{align*}

The corresponding lattice is showed in the upper left part of Figure \ref{Anti}.\\
Similarly to the proof of Proposition \ref{preradiso}, we can prove that all preradicals of the category $\Lambda\mathrm{-mod}$ are in one-to-one correspondence with $\Lambda\mathrm{-pr}$. In other words, 
$$Pr(\Lambda\mathrm{-mod})\cong \Lambda\mathrm{-pr}\cong Pr(\Lambda\mathrm{-Mod})$$

On the other hand, if $\mathbb{A}_2^{op}:\xymatrix{e_1 & e_2\ar[l]}$ is the opposite quiver, then the opposite path algebra is $\Lambda^{op}=K\mathbb{A}_2^{op}$. See \cite[III.1]{AusBook}. We have an equivalence of categories:
$$(\Lambda\mathrm{-mod})^{op}\simeq \Lambda^{op}\mathrm{-mod}$$
which induces, by Corollary \ref{equiviso}, an isomorphism:  
$$Pr((\Lambda\mathrm{-mod})^{op})\cong Pr(\Lambda^{op}\mathrm{-mod})$$

We have also the duality assignment, which is an anti-isomorphism:
$$\Delta:Pr(\Lambda\mathrm{-Mod})\longrightarrow Pr((\Lambda\mathrm{-Mod})^{op})$$

\begin{figure}[t]

\[ \scalebox{.8}{
\begin{tikzpicture}
%\begin{scope}
%\fill[color=yellow] (-3,-3) rectangle (1.5,3);
%\end{scope}

%\draw[step=0.5cm,color=blue, dotted]

\node (A1) at (1,1) {$\varheartsuit$};
\node (A2) at (0,2) {$\bullet$};
\node (A3) at (2,2) {$\varheartsuit$};
\node (A4) at (1,3) {$\bullet$};
\node (A5) at (-1,3) {$\varheartsuit$};
\node (A6) at (0,4) {$\varheartsuit$};
\node (A7) at (2,4) {$\varheartsuit$};
\node (A8) at (1,5) {$\varheartsuit$};

\node (B1) at (1,0.8) {${0}$};
\node (B2) at (0,1.8) {$\rho_{1}$};
\node (B3) at (2,1.8) {$\gamma_{0}$};
\node (B4) at (1,2.8) {$\xi$};
\node (B5) at (-1,2.6) {$\omega_{0}^{S_{2}}$};
\node (B6) at (0,3.8) {$\iota_{0}$};
\node (B7) at (2,3.8) {$\gamma_{1}$};
\node (B8) at (1,4.8) {$1$};

\draw [-, thick] (1,1) -- (-1,3);
\draw [-, thick] (2,2) -- (0,4);
\draw [-, thick] (2,4) -- (1,5);
%%%
\draw [-, thick] (-1,3) -- (1,5);
\draw [-, thick] (0,2) -- (2,4);
\draw [-, thick] (1,1) -- (2,2);

%\draw [-,dotted, thick] (-2,0) -- (13,0);
\draw [->, thick] (0.5,0.5) -- (0.5,-0.8);
\draw [->, thick] (8.5,0.5) -- (8.5,-0.8);
\node (Z1) at (1,0) {$\Delta$};
\node (Z2) at (9,0) {$\Delta'$};

%%%%%%%%%%%%%%%%%%%%%%%%%%%%%%%
\node (M1) at (9,1) {$\varheartsuit$};
\node (M2) at (10,2) {$\varheartsuit$};
\node (M3) at (8,2) {$\varheartsuit$};
\node (M4) at (9,3) {$\varheartsuit$};
\node (M5) at (11,3) {$\varheartsuit$};
\node (M6) at (10,4) {$\varheartsuit$};

\node (N1) at (9,0.8) {$0$};
\node (N2) at (10,1.8) {$\gamma_{0}$};
\node (N3) at (8,1.6) {$\omega_{0}^{S_{2}}$};
\node (N4) at (9,2.8) {$\iota_{0}$};
\node (N5) at (11,2.8) {$\gamma_{1}$};
\node (N6) at (10,3.8) {$1$};

\draw [-, thick] (9,1) -- (8,2);
\draw [-, thick] (9,1) -- (11,3);
\draw [-, thick] (8,2) -- (10,4);
\draw [-, thick] (10,4) -- (11,3);
\draw [-, thick] (9,3) -- (10,2);

%\node (X8) at (2,0.5) {$\text{reflexion}$};

%%%%%%
%%%
%%%%%
%%%%

\node (C1) at (1,-5) {$\spadesuit$};
\node (C2) at (0,-4) {$\spadesuit$};
\node (C3) at (2,-4) {$\spadesuit$};
\node (C4) at (1,-3) {$\bullet$};
\node (C5) at (-1,-3) {$\spadesuit$};
\node (C6) at (0,-2) {$\bullet$};
\node (C7) at (2,-2) {$\spadesuit$};
\node (C8) at (1,-1) {$\spadesuit$};

\node (D1) at (1,-5.3) {${\underline{0}}$};
\node (D2) at (0,-4.3) {$\underline{\rho_{1}}$};
\node (D3) at (2,-4.3) {$\underline{\gamma_{0}}$};
\node (D4) at (1,-3.3) {$\underline{\xi}$};
\node (D5) at (-1,-3.4) {$\underline{\omega_{0}^{S_{2}}}$};
\node (D6) at (0,-2.3) {$\underline{\iota_{0}}$};
\node (D7) at (2,-2.3) {$\underline{\gamma_{1}}$};
\node (D8) at (1,-1.3) {$\underline{1}$};

\draw [-, thick] (1,-5) -- (-1,-3);
\draw [-, thick] (2,-4) -- (0,-2);
\draw [-, thick] (2,-2) -- (1,-1);
%%%
\draw [-, thick] (-1,-3) -- (1,-1);
\draw [-, thick] (0,-4) -- (2,-2);
\draw [-, thick] (1,-5) -- (2,-4);

\node (P1) at (9,-1) {$\spadesuit$};
\node (P2) at (10,-2) {$\spadesuit$};
\node (P3) at (8,-2) {$\spadesuit$};
\node (P4) at (9,-3) {$\spadesuit$};
\node (P5) at (11,-3) {$\spadesuit$};
\node (P6) at (10,-4) {$\spadesuit$};

\node (Q1) at (9,-1.3) {$\underline{1}$};
\node (Q2) at (10,-2.3) {$\underline{\gamma_{1}}$};
\node (Q3) at (8,-2.4) {$\underline{\omega_{0}^{S_{2}}}$};
\node (Q4) at (9,-3.3) {$\underline{\rho_{1}}$};
\node (Q5) at (11,-3.3) {$\underline{\gamma_{0}}$};
\node (Q6) at (10,-4.3) {$\underline{0}$};

\draw [-, thick] (9,-1) -- (8,-2);
\draw [-, thick] (9,-1) -- (11,-3);
\draw [-, thick] (8,-2) -- (10,-4);
\draw [-, thick] (10,-4) -- (11,-3);
\draw [-, thick] (9,-3) -- (10,-2);

\node (F1) at (4,4) {$Pr(\Lambda\mathrm{-Mod})$};
\node (F2) at (4,-5) {$Pr((\Lambda\mathrm{-Mod})^{op})$};

\node (S1) at (12,1.5) {$Idem(\Lambda\mathrm{-Mod})$};

\node (S2) at (12,-2) {$Rad((\Lambda\mathrm{-Mod})^{op})$};

\draw [left hook->, thick] (7,3) -- (3,3);

%\node (W1) at (8,3) {$^{\supset}$};

\draw [left hook->, thick] (7,-3) -- (3,-3);

\end{tikzpicture}

}\]
\caption{\label{Anti} $K\mathbb{A}_2\mathrm{-pr}$ and $Pr((K\mathbb{A}_2\mathrm{-mod})^{op})$
} 
\end{figure}

And which restricted to the set of idempotent preradicals in $Pr(\Lambda-Mod)$ induces an anti-isomorphism onto the set of radicals in $Pr((\Lambda-Mod)^{op})$:
$$\Delta':Idem(\Lambda\mathrm{-Mod})\longrightarrow Rad((\Lambda\mathrm{-Mod})^{op})$$
\end{example}

\begin{example}

Let $\mathcal{A}$ be the category of torsion abelian groups.
It is well known that $\mathcal{A}$ is a complete and cocomplete abelian category (see \cite[Exercise 4, p. 91]{MitBook}). Moreover, $\mathcal{A}$ has no nonzero projective objects and hence it is not equivalent to any complete category of modules.\\

Let $\mathbb{P}$ be the set of all prime numbers. Following \cite{Dickson}, a \textbf{restricted Steinitz number} is a formal symbol $\prod_{p\in \mathbb{P}} p^{e(p)}$ where $e(p)$ is $0$, $1$ or $\infty$. A class of abelian groups $\mathcal{T}$ is a \textbf{torsion class} if it is closed under epimorphisms, arbitrary coproducts and extensions. Let $\mathrm{Tors}(\mathcal{A})$ be the collection of all torsion classes contained in $\mathcal{A}$. Let $\mathbb{STE}$ be the set of all restricted Steinitz numbers. Dickson proves (\cite[Theorem 2.6]{Dickson}) that there is a one to one correspondence:
$$\Theta:\mathbb{STE}\longrightarrow \mathrm{Tors}(\mathcal{A})$$

Since $\mathrm{Tors}(\mathcal{A})$ is a partial order by inclusion of classes and, furthermore, it is a complete lattice, then $\Theta$ induces on $\mathbb{STE}$ a partial order which is also a complete lattice. Thus, $\Theta$ is a lattice isomorphism. The induced partial order in $\mathbb{STE}$ is described as follows:
$$\prod_{p\in \mathbb{P}} p^{e(p)}\leq \prod_{p\in \mathbb{P}} p^{e'(p)}\Longleftrightarrow\text{ for each } p\in\mathbb{P},\ e(p)\preceq e'(p)$$
\noindent where $\preceq$ is the linear order on $\{0,\infty, 1\}$ such that $0\preceq\infty\preceq 1$. Therefore, if $\mathbf{C_3}$ is the chain with three elements, then $\mathbb{STE}$ is isomorphic as a lattice to the direct product $\mathbf{C_3}^{\omega}$. Notice that this lattice is autodual.

We have also a bijective correspondence (see \cite[Chapter VI, section 2]{Stentrom}):
$$\Gamma: \mathrm{Tors}(\mathcal{A})\longrightarrow Radidem(\mathcal{A})$$
\noindent which is also a lattice isomorphism. Notice that, in particular, we conclude that both $\mathrm{Tors}(\mathcal{A})$ and $Radidem(\mathcal{A})$ are sets.\\

On the other hand, we have the Pontryagin duality (see for example \cite{Armacost})
$$D=\mathrm{Hom}_{\mathbf{Ab}}(-,\mathbb{R}/\mathbb{Z}): \mathrm{LCA}\longrightarrow \mathrm{LCA}^{op}$$
where $\mathrm{LCA}$ denotes the category of locally compact abelian groups (with continuous morphisms). This duality induces by restriction the duality:

$$D'=\mathrm{Hom}_{\mathbf{Ab}}(-,\mathbb{R}/\mathbb{Z}): \mathcal{A}\longrightarrow \mathrm{Profin}$$
where $\mathrm{Profin}$ is the category of abelian profinite groups (see \cite[Theorem 2.9.6, p.64]{Ribes}). Recall that the category of abelian profinite groups is equivalent to the the category of abelian totally disconnected groups (see \cite[Proposition 0, example 4, p.3]{Serre}). Therefore we have the equivalence $\mathcal{A}^{op}\simeq \mathrm{Profin}$.
By Corollary \ref{antisos}, there is an anti-isomorphism $\Delta''':Radidem(\mathcal{A})\longrightarrow Radidem(\mathcal{A}^{op})$. Since $\mathbf{C_3}^{\omega}$ is an autodual lattice, then we have lattice isomorphisms:
$$Radidem(\mathcal{A})\cong Radidem(\mathcal{A}^{op})\cong Radidem(\mathrm{Profin}) \cong\mathbf{C_3}^{\omega}$$

\end{example}
\vspace{.5in}
%%%%%%%%%%%%%%%%%%%%%

\textbf{Funding Declaration:} The authors would like to thank the SECIHTI Project CBF-2023-2024-2630, and the Research Support Office of the Metropolitan Autonomous University (DAI UAM) for the support granted.

%\textbf{Competing Interest Declaration:} There are no competing interests.

%\textbf{Availability of Data and Materials Declaration:} This manuscript does not generate any data or materials, so that there is no such availability.

\textbf{Author Contribution Statement:} All authors, Rogelio Fern\'andez-Alonso, Janeth Maga\~na, Martha Lizbeth Shaid Sandoval-Miranda and Valente Santiago-Vargas, contributed equally to this work.

\vspace{.5in}

\bibliographystyle{sn-chicago}
\bibliography{GLFC}% common bib file

\end{document}